\documentclass{article}

\usepackage[english]{babel}

\usepackage[a4paper,top=2.5cm,bottom=2.5cm,left=2.6cm,right=2.6cm]{geometry}

\usepackage{amsmath,amssymb,amsthm}
\usepackage{graphicx}
\usepackage[colorlinks=true, allcolors=blue]{hyperref}
\usepackage{enumitem}
\newlist{typesofedges}{enumerate}{1}
\setlist[typesofedges, 1]
{label=(Type \arabic{typesofedgesi}),
leftmargin=45pt,
rightmargin=0pt
}

\newlist{typesofedgestw}{enumerate}{1}
\setlist[typesofedgestw, 1]
{label=(Type \arabic{typesofedgestwi}),
leftmargin=45pt,
rightmargin=0pt,
start =0
}

\usepackage{comment}
\usepackage{tikz}
\usetikzlibrary {decorations.pathmorphing, decorations.pathreplacing, decorations.shapes, decorations.markings, arrows.meta, calc, intersections}
\usepackage[font=small,labelfont=bf,margin=.5cm]{caption}
\usepackage{subcaption}

\usepackage{etoolbox}
\makeatletter
\patchcmd{\@endtheorem}{\@endpefalse}{}{}{}
\makeatother

\newtheorem{theorem}{Theorem}

\newtheorem{remark}[theorem]{Remark}
\newtheorem{observation}[theorem]{Observation}
\newtheorem{lemma}[theorem]{Lemma}
\newtheorem{definition}[theorem]{Definition}

\newtheorem{conjecture}[theorem]{Conjecture}

\newcommand{\wavytriangle}{%
  \mathord{%
    \tikz[baseline=0.3ex, scale=0.17]{
      \draw (0,0) -- (1,1.6);
      \draw (1,1.6) -- (1.8,0.4);
      \draw[decorate,
            decoration={snake, amplitude=0.8pt, segment length=5pt}]
        (0,0) -- (1.8,0.4);
    }%
  }%
}

\author{Julian Becker\thanks{Department of Mathematics, LMU Munich, Munich 80333 Germany. Email: \href{mailto:becker@math.lmu.de}{\nolinkurl{becker@math.lmu.de}.}}
\and
Konstantinos Panagiotou\thanks{Department of Mathematics, LMU Munich, Munich 80333 Germany. Email: \href{mailto:kpanagio@math.lmu.de}{\nolinkurl{kpanagio@math.lmu.de}.}}
\and
Matija Pasch\thanks{Munich, Germany. E-mail: \href{mailto:matija.pasch@gmail.com}{\nolinkurl{matija.pasch@gmail.com}}.}
}

\title{Improved Universal Graphs for Trees\footnote{An extended abstract that contains parts of this work was published in EuroComb'25~\cite{eurocomb25}.}}

\date{\today}

\begin{document}

\maketitle

\begin{abstract}
A graph $G$ is \emph{universal} for a class of graphs $\cal C$, if, up to isomorphism, $G$ contains every graph in $\cal C$ as a subgraph. In 1978, Chung and Graham
asked for the minimal number $s(n)$ of edges in a graph with $n$ vertices that is universal for all trees with $n$ vertices. The currently best bounds assert that $n\ln n-O(n)\le s(n) \le C n\ln n+O(n)$, where $C = \frac{14}{5\ln 2} \approx 4.04$. We improve the upper bound to $c n\ln n + O(n)$, where $c = \frac{19}{6\ln 3} \approx 2.88$. In the proof we develop a strategy that, broadly speaking, is based on separating trees into \emph{three} parts, thus enabling us to embed them in a structure that originates from ternary trees.

Our method also applies to graphs with a bound on their treewidth. Let $s_w(n)$ be the minimum number of edges in a $n$-vertex graph that is universal for graphs with treewidth~$w$. By performing a blow-up to our universal structure for trees we establish that $nw \ln(n/w) -O(nw) \leq  s_w(n) \leq \frac{19}{6\ln3} n (w+1) \ln(n/w) + O(nw)$.
\end{abstract}

\section{Introduction}
A graph $G$ is \emph{universal} for a class of graphs $\cal C$, if, up to isomorphism, $G$ contains every graph in $\cal C$ as a subgraph. The study of universal graphs was initiated by Rado in 1964 \cite{rado1964} and has flourished over the last six decades, see~\cite{ferber2016,johannsen2013,alon2000} for some recent work and further references. In this paper we consider universal graphs with $n$ vertices for the class  of trees with $n$
vertices, also known as \emph{tree-complete} graphs~\cite{sedlavcek1975,nebesky1975}.
Chung and Graham asked for the minimum number $s(n)$ of edges in a tree-complete graph in 1978~\cite{chung1978}. In the same paper, by bounding the degree sequence of any such  graph, they established the lower bound $s(n) \geq\frac{1}{2}n\ln n -O(n)$. Moreover, using a recursive construction and improving over several previous results, they showed that $s(n)\leq \frac{7}{\ln 4}n \ln n+O(n)$ in \cite{chung1983}. 
In the following forty years, neither of the bounds was improved.
However, recently the lower bound was pushed to $n\ln n-O(n)$~\cite{gyori2025}, where, instead of only considering the degree sequence, the authors counted the edges carefully and jointly for all embeddings.  The currently best upper bound is by Kaul et.~al.~\cite{kaul2025universalgraphstreestreelike}, who showed that $s(n)\le\frac{14}{5\ln2}n \ln n +O(n)$, and who improved upon~\cite{kim2025}. Hence, to date,
\[
    n\ln n - O(n) \leq s(n) \leq \frac{14}{5\ln 2}n \ln n+O(n),
\]
where the constant is numerically $4.039\dots$.
Our main result improves this to $19/(6\ln 3) = 2.882\dots$.
\begin{theorem}\label{thm:main}
We have $s(n)\le\frac{19}{6\ln3}n\ln n+O(n)$.
\end{theorem}
In the proof we develop a strategy for embedding trees iteratively into a host graph based on ternary trees, which, in its very essence, is inspired by the aforementioned~\cite{chung1983}. In particular, given a certain tree, we remove carefully up to two vertices to achieve a partition into components of suitable sizes. 
Crucially, we are able to embed these components recursively, while preserving  certain properties of the host graph. As we shall demonstrate in Section~\ref{sec: binary trees}, we are able to simplify and improve upon along these lines also the original strategy presented in \cite{chung1983}, which is based on binary trees.

The improved upper bound is one further step towards answering Chung and Graham's question about determining $s(n)$. However, obtaining \emph{asymptotic} bounds would already be of great interest, and a starting point could be to show the following statement that looks very natural in this context.

\begin{conjecture}
The $\lim_{n\to \infty} \frac{s(n)}{n\ln n}$ exists.
\end{conjecture}
If this conjecture is true, then it would be interesting to determine the value $c^*$ of the limit. The results in \cite{gyori2025} and Theorem~\ref{thm:main} assert that if the limit exists, then $1 \le c^* \le \frac{19}{6\ln 3}$. We have no indication that either of the bounds could be tight.
Indeed, the lower bound is obtained by degree considerations (every universal graph must contain vertices of degree at least $n/i$ for very many $i$), while the upper bounds (here and everywhere else) are derived  by cutting the trees in smaller parts and embedding them recursively in the universal structure. We expect that some interaction actually occurs between the two different parameters, so that eventually both contribute to $s(n)$.

Our results for trees can be generalized to graphs that are structurally close to trees. The framework for this is given by the notion of treewidth, introduced by Robertson and Seymour in \cite{ROBERTSONtreewidth}, which, roughly speaking, measures how similar a graph is to a tree. A \textit{tree-decomposition} of a graph $G$ is a collection of subsets $(B_x:x\in T)$ of the vertex set $V(G)$ (called \textit{bags}) indexed by the vertices of a tree $T$, such that
\begin{itemize}
\itemsep0pt
    \item every vertex $v\in G$ appears in at least one bag;
    \item for every edge $uv$ in $G$, there exists some bag containing both $u$ and $v$, and
    \item for every vertex $v\in G$, the set of bags containing $v$ form a subtree of $T$.    
\end{itemize}
The \textit{width} of a tree-decomposition is the size of the largest bag minus one. The treewidth of $G$, denoted by $\text{tw}(G)$, is the minimum width over all tree-decomposition of $G$. 
Let $s_w(n)$ be the minimum number of edges in a $n$-vertex graph that is universal for $n$-vertex graphs with treewidth~$w$. Since trees have treewidth one, we have in particular $s_1(n)=s(n)$.
In \cite{kaulwood2025}, Kaul and Wood establish that
\[
s_w(n) \geq \frac{1}{2}nw \ln(n/w) -O(nw).
\]
With a slight modification of their argument, we sharpen the bound by a factor of two. Moreover, by applying a graph blow-up to our universal construction for trees, we obtain a graph matching the asymptotic order of the lower bound. The following theorem summarizes our results.
\begin{theorem}\label{thm:main-treewidth}
Uniformly in $w$ we have $nw \ln(n/w) -O(nw) \leq  s_w(n) \leq \frac{19}{6\ln3} (w+1) n \ln(n/w) + O(nw)$.
\end{theorem}
The factor $w+1$ in the upper bound arises because, in general, every minimal vertex separator that splits a graph of treewidth $w$ into three parts has size $w+1$. In the special case $w=1$, trees behave differently: every minimal separator that partitions a tree, except for the path, into three parts has size one rather than two. This explains the sharper result for $s(n)$. Finally, let us remark that also in the setting of graphs with a given treewidth, it would be interesting to determine $\lim_{n\to \infty}\frac{s_w(n)}{nw\ln(n/w)}$, or at least, to justify its existence.

\section{Basic constructions \& proof overview}\label{sec:proof-strategy}
We establish Theorem \ref{thm:main} by explicitly constructing universal graphs with $\frac{19}{6\ln3}n\ln n+O(n)$ edges. As mentioned in the introduction, we use ternary trees as guiding building structures for the universal graphs. In order to make the construction more accessible, we also explore a simpler construction on binary trees that already improves upon~\cite{chung1983}. We introduce some notation first. For $d\ge 2$ consider the full $d$-ary tree~$T_{h,d}$ of height $h\ge -1$ with levels~$0$ to~$h$ (if $h \ge 0$) on the vertex set $V_{h,d}$ given by
\[
V_{-1,d}=\emptyset
\quad
\text{and}
\quad
V_{h,d}=\bigcup_{0\leq \ell\leq h}\{1,2,\dots,d\}^\ell,
\quad 
\text{for }
h\geq 0.
\]
For a level $0\le \ell<h$
and a vertex $v\in\{1,\dots,d\}^\ell$, the \emph{children} of~$v$ are $v1, \dots, vd\in\{1,\dots,d\}^{\ell+1}$, where we use an economic notation for elements of $\{1,\dots,d\}^\ell$, e.g.~$13231\in\{1,2,3\}^5$. 
Additionally, for~$v\in V_{h,d}$, let $D_v$ denote the set of all \emph{descendants} of $v$,~i.e., all vertices in $V_{h,d}\setminus \{v\}$ with prefix~$v$. 

For $a \in\mathbb{N}$, define $v\pm a$ as the $a$-th successor/predecessor of $v$ using the lexicographical order imposed on the level of $v$, and using the rules $\text{succ}(d\dots d) := 1\dots1$ and $\text{pred}(1\dots1):=d\dots d$.
For example, if $v=d(d-1)$, then $v+1=dd$, $v+2=11$ and if $v=dd$, then $v-d = (d-1)d$.
We compare vertices at different levels lexicographically by attaching $0$'s to the shorter one, where $0$ is  the smallest character in the alphabet.
Moreover, for $T_{h,d}$ we introduce an additional ordering on the vertices such that on every full $d-$ary (sub-)tree the root comes last.
Precisely, we define the relation $``\succ"$ as the reversed lexicographical order, that is for $x,y \in V_{h,d}$, $x\neq y$, we set $x\succ y$ if $x$ is lexicographically smaller than~$y$.
In the proof we will show that a certain graph $U'$ with vertices in $V_{h,d}$ is universal for the class of forests. In particular, we will do so by finding an isomorphic subgraph of the first tree $T'$ in the given forest at the \emph{smallest} $|V(T')|$ vertices of $U'$ (with respect to $\succ$). After that, we find an isomorphic subgraph of the second tree $T''$ at the smallest $|V(T'')|$ vertices of $U'\setminus T'$, and so forth. This process can be interpreted as the graph $U'$ being eaten up by the trees, hence we call $\succ$ the \emph{eating order}. Figure~\ref{fig:ternarytree} shows $T_{3,3}$ and  the eating order.

\begin{figure}[tb]
    \centering
\begin{tikzpicture}[font=\sffamily][scale = 0.85]

"\clip (0cm,0cm) rectangle (14.8cm,6cm);
\coordinate (v0) at (7.4,5.5);
\coordinate (v1) at (2.62307692307692,4.06666666666667);
\coordinate (v2) at (7.4,4.06666666666667);
\coordinate (v3) at (12.1769230769231,4.066666666666677);
\coordinate (v11) at (1.03076923076923,2.63333333333333);
\coordinate (v12) at (2.62307692307692,2.63333333333333);
\coordinate (v13) at (4.21538461538462,2.63333333333333);
\coordinate (v21) at (5.80769230769231,2.63333333333333);
\coordinate (v22) at (7.4,2.63333333333333);
\coordinate (v23) at (8.99230769230769,2.63333333333333);
\coordinate (v31) at (10.5846153846154,2.633333333333333);
\coordinate (v32) at (12.1769230769231,2.633333333333333);
\coordinate (v33) at (13.7692307692308,2.63333333333333);
\coordinate (v111) at (0.5,1);
\coordinate (v112) at (1.03076923076923,1);
\coordinate (v113) at (1.56153846153846,1);
\coordinate (v121) at (2.09230769230769,1);
\coordinate (v122) at (2.62307692307692,1);
\coordinate (v123) at (3.15384615384615,1);
\coordinate (v131) at (3.68461538461538,1);
\coordinate (v132) at (4.21538461538462,1);
\coordinate (v133) at (4.74615384615385,1);
\coordinate (v211) at (5.27692307692308,1);
\coordinate (v212) at (5.80769230769231,1);
\coordinate (v213) at (6.33846153846154,1);
\coordinate (v221) at (6.86923076923077,1);
\coordinate (v222) at (7.4,1);
\coordinate (v223) at (7.93076923076923,1);
\coordinate (v231) at (8.46153846153846,1);
\coordinate (v232) at (8.99230769230769,1);
\coordinate (v233) at (9.52307692307692,1);
\coordinate (v311) at (10.0538461538462,1);
\coordinate (v312) at (10.5846153846154,1);
\coordinate (v313) at (11.1153846153846,1);
\coordinate (v321) at (11.6461538461538,1);
\coordinate (v322) at (12.1769230769231,1);
\coordinate (v323) at (12.7076923076923,1);
\coordinate (v331) at (13.2384615384615,1);
\coordinate (v332) at (13.7692307692308,1);
\coordinate (v333) at (14.3,1);
\draw (v0) -- (v1);
\draw (v0) -- (v2);
\draw (v0) -- (v3);
\draw (v1) -- (v11);
\draw (v1) -- (v12);
\draw (v1) -- (v13);
\draw (v2) -- (v21);
\draw (v2) -- (v22);
\draw (v2) -- (v23);
\draw (v3) -- (v31);
\draw (v3) -- (v32);
\draw (v3) -- (v33);
\draw (v11) -- (v111);
\draw (v11) -- (v112);
\draw (v11) -- (v113);
\draw (v12) -- (v121);
\draw (v12) -- (v122);
\draw (v12) -- (v123);
\draw (v13) -- (v131);
\draw (v13) -- (v132);
\draw (v13) -- (v133);
\draw (v21) -- (v211);
\draw (v21) -- (v212);
\draw (v21) -- (v213);
\draw (v22) -- (v221);
\draw (v22) -- (v222);
\draw (v22) -- (v223);
\draw (v23) -- (v231);
\draw (v23) -- (v232);
\draw (v23) -- (v233);
\draw (v31) -- (v311);
\draw (v31) -- (v312);
\draw (v31) -- (v313);
\draw (v32) -- (v321);
\draw (v32) -- (v322);
\draw (v32) -- (v323);
\draw (v33) -- (v331);
\draw (v33) -- (v332);
\draw (v33) -- (v333);
\draw[fill=white] (v0) circle (0.24cm);
\draw[fill=white] (v1) circle (0.24cm);
\draw[fill=white] (v2) circle (0.24cm);
\draw[fill=white] (v3) circle (0.24cm);
\draw[fill=white] (v11) circle (0.24cm);
\draw[fill=white] (v12) circle (0.24cm);
\draw[fill=white] (v13) circle (0.24cm);
\draw[fill=white] (v21) circle (0.24cm);
\draw[fill=white] (v22) circle (0.24cm);
\draw[fill=white] (v23) circle (0.24cm);
\draw[fill=white] (v31) circle (0.24cm);
\draw[fill=white] (v32) circle (0.24cm);
\draw[fill=white] (v33) circle (0.24cm);
\draw[fill=white] (v111) circle (0.24cm);
\draw[fill=white] (v112) circle (0.24cm);
\draw[fill=white] (v113) circle (0.24cm);
\draw[fill=white] (v121) circle (0.24cm);
\draw[fill=white] (v122) circle (0.24cm);
\draw[fill=white] (v123) circle (0.24cm);
\draw[fill=white] (v131) circle (0.24cm);
\draw[fill=white] (v132) circle (0.24cm);
\draw[fill=white] (v133) circle (0.24cm);
\draw[fill=white] (v211) circle (0.24cm);
\draw[fill=white] (v212) circle (0.24cm);
\draw[fill=white] (v213) circle (0.24cm);
\draw[fill=white] (v221) circle (0.24cm);
\draw[fill=white] (v222) circle (0.24cm);
\draw[fill=white] (v223) circle (0.24cm);
\draw[fill=white] (v231) circle (0.24cm);
\draw[fill=white] (v232) circle (0.24cm);
\draw[fill=white] (v233) circle (0.24cm);
\draw[fill=white] (v311) circle (0.24cm);
\draw[fill=white] (v312) circle (0.24cm);
\draw[fill=white] (v313) circle (0.24cm);
\draw[fill=white] (v321) circle (0.24cm);
\draw[fill=white] (v322) circle (0.24cm);
\draw[fill=white] (v323) circle (0.24cm);
\draw[fill=white] (v331) circle (0.24cm);
\draw[fill=white] (v332) circle (0.24cm);
\draw[fill=white] (v333) circle (0.24cm);
\node at (v0) {\scriptsize$\emptyset$};
\node at (v1) {\scriptsize$1$};
\node at (v2) {\scriptsize$2$};
\node at (v3) {\scriptsize$3$};
\node at (v11) {\scriptsize$11$};
\node at (v12) {\scriptsize$12$};
\node at (v13) {\scriptsize$13$};
\node at (v21) {\scriptsize$21$};
\node at (v22) {\scriptsize$22$};
\node at (v23) {\scriptsize$23$};
\node at (v31) {\scriptsize$31$};
\node at (v32) {\scriptsize$32$};
\node at (v33) {\scriptsize$33$};
\node at (v111) {\scriptsize$111$};
\node at (v112) {\scriptsize$112$};
\node at (v113) {\scriptsize$113$};
\node at (v121) {\scriptsize$121$};
\node at (v122) {\scriptsize$122$};
\node at (v123) {\scriptsize$123$};
\node at (v131) {\scriptsize$131$};
\node at (v132) {\scriptsize$132$};
\node at (v133) {\scriptsize$133$};
\node at (v211) {\scriptsize$211$};
\node at (v212) {\scriptsize$212$};
\node at (v213) {\scriptsize$213$};
\node at (v221) {\scriptsize$221$};
\node at (v222) {\scriptsize$222$};
\node at (v223) {\scriptsize$223$};
\node at (v231) {\scriptsize$231$};
\node at (v232) {\scriptsize$232$};
\node at (v233) {\scriptsize$233$};
\node at (v311) {\scriptsize$311$};
\node at (v312) {\scriptsize$312$};
\node at (v313) {\scriptsize$313$};
\node at (v321) {\scriptsize$321$};
\node at (v322) {\scriptsize$322$};
\node at (v323) {\scriptsize$323$};
\node at (v331) {\scriptsize$331$};
\node at (v332) {\scriptsize$332$};
\node at (v333) {\scriptsize$333$};
\node at ($(v0)+(0.45,0)$) {\scriptsize$\mathit{40}$};
\node at ($(v1)+(0.4,-0.1)$) {\scriptsize$\mathit{39}$};
\node at ($(v2)+(0.4,-0.1)$) {\scriptsize$\mathit{26}$};
\node at ($(v3)+(0.4,-0.1)$) {\scriptsize$\mathit{13}$};
\node at ($(v11)+(0.33,-0.3)$) {\scriptsize$\mathit{38}$};
\node at ($(v12)+(0.33,-0.3)$) {\scriptsize$\mathit{34}$};
\node at ($(v13)+(0.33,-0.3)$) {\scriptsize$\mathit{30}$};
\node at ($(v21)+(0.33,-0.3)$) {\scriptsize$\mathit{25}$};
\node at ($(v22)+(0.33,-0.3)$) {\scriptsize$\mathit{21}$};
\node at ($(v23)+(0.33,-0.3)$) {\scriptsize$\mathit{17}$};
\node at ($(v31)+(0.33,-0.3)$) {\scriptsize$\mathit{12}$};
\node at ($(v32)+(0.33,-0.3)$) {\scriptsize$\mathit{8}$};
\node at ($(v33)+(0.33,-0.3)$) {\scriptsize$\mathit{4}$};
\node at ($(v111)+(0,-0.4)$) {\scriptsize$\mathit{37}$};
\node at ($(v112)+(0,-0.4)$) {\scriptsize$\mathit{36}$};
\node at ($(v113)+(0,-0.4)$) {\scriptsize$\mathit{35}$};
\node at ($(v121)+(0,-0.4)$) {\scriptsize$\mathit{33}$};
\node at ($(v122)+(0,-0.4)$) {\scriptsize$\mathit{32}$};
\node at ($(v123)+(0,-0.4)$) {\scriptsize$\mathit{31}$};
\node at ($(v131)+(0,-0.4)$) {\scriptsize$\mathit{29}$};
\node at ($(v132)+(0,-0.4)$) {\scriptsize$\mathit{28}$};
\node at ($(v133)+(0,-0.4)$) {\scriptsize$\mathit{27}$};
\node at ($(v211)+(0,-0.4)$) {\scriptsize$\mathit{24}$};
\node at ($(v212)+(0,-0.4)$) {\scriptsize$\mathit{23}$};
\node at ($(v213)+(0,-0.4)$) {\scriptsize$\mathit{22}$};
\node at ($(v221)+(0,-0.4)$) {\scriptsize$\mathit{20}$};
\node at ($(v222)+(0,-0.4)$) {\scriptsize$\mathit{19}$};
\node at ($(v223)+(0,-0.4)$) {\scriptsize$\mathit{18}$};
\node at ($(v231)+(0,-0.4)$) {\scriptsize$\mathit{16}$};
\node at ($(v232)+(0,-0.4)$) {\scriptsize$\mathit{15}$};
\node at ($(v233)+(0,-0.4)$) {\scriptsize$\mathit{14}$};
\node at ($(v311)+(0,-0.4)$) {\scriptsize$\mathit{11}$};
\node at ($(v312)+(0,-0.4)$) {\scriptsize$\mathit{10}$};
\node at ($(v313)+(0,-0.4)$) {\scriptsize$\mathit{9}$};
\node at ($(v321)+(0,-0.4)$) {\scriptsize$\mathit{7}$};
\node at ($(v322)+(0,-0.4)$) {\scriptsize$\mathit{6}$};
\node at ($(v323)+(0,-0.4)$) {\scriptsize$\mathit{5}$};
\node at ($(v331)+(0,-0.4)$) {\scriptsize$\mathit{3}$};
\node at ($(v332)+(0,-0.4)$) {\scriptsize$\mathit{2}$};
\node at ($(v333)+(0,-0.4)$) {\scriptsize$\mathit{1}$};
\node at ($(v31)+(-0.2,0.7)$) {\scriptsize$32-2$};
\node at ($(v33)+(0,1.15)$) {\scriptsize$32+2$};"

\draw[dashed,->] ($(v32)+(-0.18,0.18)$) to[bend right=30] ($(v23)+(0.18,0.18)$);

\draw[dashed] ($(v32)+(0.18,0.18)$) to[bend left=15] ($(v33)+(0.4,0.6)$);
\draw[dashed] ($(v33)+(0.4,0.6)$) to[bend right=80] ($(v33)+(0.4,1.0)$);
\draw[dashed] ($(v33)+(0.4,1.0)$) to ($(v33)+(0,1.0)$);
\node at ($(v33)+(-0.2,1.0)$) {$\dots$};
\draw[dashed] ($(v11)+(0.4,1.0)$) to ($(v11)+(-0.5,1.0)$);
\draw[dashed] ($(v11)+(-0.5,1.0)$) to[bend right=70] ($(v11)+(-0.5,0.5)$);
\draw[dashed,->] ($(v11)+(-0.5,0.5)$) to ($(v11)+(-0.18,0.18)$);

\draw[loosely dotted, rounded corners=9pt,line width=0.35mm] ($(v0)+(0.7,0.4)$) -- ($(v0)+(-0.5,0.4)$) -- ($(v1)+(-0.4,0.5)$) -- ($(v11)+(-0.4,0.5)$) -- ($(v111)+(-0.4,0.5)$) -- ($(v111)+(-0.4,-0.7)$) -- ($(v133)+(0.26,-0.7)$) -- ($(v133)+(0.26,0.5)$) -- ($(v13)+(0.6,0.3)$) --  ($(v1)+(0.7,-0.2)$) -- ($(v0)+(0.7,-0.2)$) -- cycle;

\end{tikzpicture}
\caption{\small The ternary tree $T_{3,3}$ on vertex set $V_{3,3}$. The vertex labels are denoted within the circles and the position of each vertex in the eating order is next to the circle. For example, the vertex $32$ is the $8$-th vertex to be eaten. Further, the arithmetic operations $32+2=11$ and $32-2=23$ are indicated by the dashed arrows.
Moreover, the eating order on $V_{3,3}$ also determines the vertex set of $U_{n,3}$ of height $3$. For instance, the vertex set $V(U_{14,3})$ is visualized by the dotted curve ($26$ vertices were eaten). 
}
\label{fig:ternarytree}
\end{figure}

With this notation at hand, we construct for each $h$ a graph $T^*_{h,d}$ on $V_{h,d}$ by adding for every $v\in V_{h,d}$ all edges that contain $v$ and
\begin{typesofedges}
    \item\label{en: type 1 edges} every descendant of $v$, i.e., every vertex in $D_v$; 
    \item\label{en: type 2 edges} $v-1, \dots, v-(d-1)$ and every vertex in $D_{v-1}\cup \dots \cup D_{v-(d-1)}$.
\end{typesofedges}
In the case $d=3$ we add a few more edges, namely those containing $v\in V_{h,3}$ and
\begin{typesofedges}[resume]
    \item\label{en: type 3 edges} every vertex in the half (rounded down) of $\{z\} \cup D_z$ that is eaten last, where $z$ is the lexicographically smallest child of $v+1$.
\end{typesofedges}
The necessity of having additional edges in the case $d=3$ arises due to a delicate detail that we will discuss later.
With $T_{h,d}^*$ at hand, we now consider the following graphs that are central in what follows. 
\begin{definition}\label{def:U}
    For $d, n \in \mathbb{N}$, let $h\ge 0$ be such that $|V_{h-1,d}|<n\le|V_{h,d}|$. We define $U_{n,d}$ as the induced subgraph of $T^*_{h,d}$ on the $n$ vertices of $T^*_{h,d}$ that are eaten last.
\end{definition}
For example, Figure \ref{fig:ternarytree} shows the vertex sets $V(U_{14,3})$ and $V(U_{40,3})$. The crucial step in our proof, and the one that requires most work, is to verify that for $d\in \{2,3\}$, in this way, we obtain universal graphs, see Sections~\ref{sec: binary trees} and \ref{sec:ternary trees}.
\begin{lemma}\label{lem:univ}
    For $n\in \mathbb{N}$ and $d\in \{2,3\}$ the graphs $U_{n,d}$ are universal.
\end{lemma}
In order to show universality, our starting point is the following simple and well-known property of separating vertices in trees \cite{chung1983}.
Let $|G| = |V(G)|$ be the number of vertices of a graph $G$. 
\begin{lemma}\label{lem:sep-chung}
    Let $t \in\mathbb{N}_0$ and $F$ be a forest with $|F|\geq t+1$. Then for some vertex $s$, there is a forest $F' \subset F \setminus s$ such that
    $t \leq |F'| \leq 2t$.
\end{lemma}
Lemma \ref{lem:sep-chung} allows to split a forest, by removing one vertex, into two forests containing a suitable number of vertices. This paves the way to construct universal graphs using the binary tree as a base structure and to pursue a divide and conquer approach for the embedding of \emph{any} forest. This is the main approach followed in~\cite{chung1983}. We extend this idea by using ternary trees as building blocks for the universal graphs, based on the following refined version of Lemma~\ref{lem:sep-chung}.
By a slight abuse of terminology, let (the possibly empty) forests $F_1,\dots,F_t$ be a \emph{partition} of a forest $F$ if they are vertex disjoint and $F_1 \cup \dots \cup F_t =F$. 
\begin{lemma}\label{lem:sep-main}
    Let $0 \leq m$, $2m \leq M$ and $F$ be a forest with $|F|\geq M+1$. Then there exists a vertex $s \in F$ and partition $F_1, F_2, F_3$ of $F\setminus s$ such that 
    \[
    m \leq |F_3| \leq M,
    \quad
    |F_1| \leq |F| -1 -M,
    \quad
    \text{and}
    \quad
    |F_2| \leq |F_1|.
    \]
\end{lemma}
With the choice $m=t$ and $M=2t$, this implies Lemma \ref{lem:sep-chung} with $F' = F_3$. However, Lemma~\ref{lem:sep-main} provides additional structural information about the remaining forest $F^*=F \setminus (s\cup F_3)$. That is, if $|F_3|<M$, then $F^*$ is the disjoint union of at least \emph{two} trees. 
We use Lemma \ref{lem:sep-main} to iteratively embed forests into certain graphs, that we call \emph{admissible}. As it will turn out, these graphs have the handy property that after removing vertices following the eating order the result is again admissible. 
\begin{definition}\label{def:admissible}
    Let $d\geq 2$. A graph $A\neq \emptyset$ is called \emph{admissible}, if there exists $h$ such that $A$ is isomorphic to the induced subgraph on the last $|A|$ vertices in the eating order of $T_{h,d}^*$. The eating order on $A$ is thereby naturally inherited from $T_{h,d}^*$.
\end{definition}
Due to the $d$-ary base structure of the graphs $T_{h,d}^*$, we can recursively describe admissible graphs.
\begin{remark}
\label{rem:admissible-recursive}
    Let $d \ge 2$ and let $A$ be an admissible graph. Then there exists $T^*_{h,d}$ for some $h$ such that one of the following holds.
    \begin{enumerate}\setlength{\itemsep}{0pt}
        \item $A$ only consists of the root of $T^*_{h,d}$, that we denote by $r_A$, and which is first in the eating order.
        \item\label{rem:admissible2} There exists an admissible subgraph $A'$ of $T^*_{h-1,d}$ such that $A$ is isomorphic to one of the $d$ possible ways shown in Figure \ref{fig:admissible-c}.
        Further, $A$ inherits the eating order from $T^*_{h,d}$ as follows. The vertices of $A'$ are eaten first, given by the order on $A'$. Next, the vertices of the up to $d-1$ copies of $T_{h-1,d}^*$ are eaten one after another, given by the eating order on $T_{h-1,d}^*$. Finally, $r_A$ is eaten.
    \end{enumerate}
\end{remark}
\begin{figure}[ht!]
    \centering
    \begin{subfigure}[t]{0.3\textwidth} 
        \centering
        \begin{tikzpicture}[scale=0.8, every node/.style={transform shape}]
        \node[draw, circle, fill=black, inner sep=2pt, label=right:{$r_A$}] (root2) at (0,-1) {};
        \node[draw, circle, fill=black, inner sep=2pt, label={[label distance=0.9cm]-90:$A'$}] (child1) at (-1.5,-2) {};
        \node[] at (-2,-4) {}; 
        \node[] at (2,-4) {}; 
        
        \draw (root2) -- (child1);
        
        \draw (-2,-4) -- (-1.5,-2)  -- (-1,-3.5) decorate [decoration={zigzag,segment length=2mm}] {-- cycle}; 
        \end{tikzpicture}
        \caption{}
        \label{fig:admissible-a}
    \end{subfigure}
    \begin{subfigure}[t]{0.3\textwidth} 
        \centering
        \begin{tikzpicture}[scale=0.8, every node/.style={transform shape}]
        \node[draw, circle, fill=black, inner sep=2pt, label=right:{$r_A$}] (root2) at (0,-1) {};

        \node[draw, circle, fill=black, inner sep=2pt, label={[label distance=1.3cm]-90:$T^*_{h-1,d}$}] (child1) at (-1.5,-2) {};
        \node[draw, circle, fill=black, inner sep=2pt, label={[label distance=0.9cm]-90:$A'$}] (child2) at (0,-2) {};
        \node[] at (2,-4) {}; 

        \draw (root2) -- (child1);
        \draw (root2) -- (child2);
        
        \draw (-2.15,-4) -- (-1.5,-2) -- (-0.85,-4) -- cycle;
        \draw (-0.5,-4) -- (0,-2)  -- (0.5,-3.5) decorate [decoration={zigzag,segment length=2mm}] { -- cycle};
        \end{tikzpicture}
        \caption{}
        \label{fig:admissible-b}
    \end{subfigure}
    \begin{subfigure}[t]{0.3\textwidth} 
        \centering
        \begin{tikzpicture}[scale=0.8, every node/.style={transform shape}]
        \node[draw, circle, fill=black, inner sep=2pt, label=right:{$r_A$}] (root2) at (0,-1) {};

        \node[draw, circle, fill=black, inner sep=2pt, label={[label distance=1.3cm]-90:$T^*_{h-1,d}$}] (child1) at (-2.8,-2) {};
        \node[draw, circle, fill=black, inner sep=2pt, label={[label distance=1.3cm]-90:$T^*_{h-1,d}$}] (child2) at (0,-2) {};
        \node[draw, circle, fill=black, inner sep=2pt, label={[label distance=0.9cm]-90:$A'$}] (child3) at (1.5,-2) {};

        
        \draw (root2) -- (child1);
        \draw (root2) -- (child2);
        \draw (root2) -- (child3);
        
        \draw (-3.45,-4) -- (-2.8,-2) -- (-2.15,-4) -- cycle; 
        \draw (-0.65,-4) -- (0,-2) -- (0.65,-4) -- cycle;  
        \draw (1,-4) -- (1.5,-2)  -- (2,-3.5) decorate [decoration={zigzag,segment length=2mm}] {-- cycle}; 

        \node[draw, circle, fill=black, inner sep=0.5pt] (k1) at (-1.6,-2) {};
        \node[draw, circle, fill=black, inner sep=0.5pt, label={[label distance=0.1cm]-90:$\leq d-1$ times}] (k2) at (-1.4,-2) {};
        \node[draw, circle, fill=black, inner sep=0.5pt] (k3) at (-1.2,-2) {};
        
        \end{tikzpicture}
        \caption{}
        \label{fig:admissible-c}
    \end{subfigure}
    \caption{\small The admissible graph $A$ is given by its root $r_A$, the admissible subgraph $A'$ of $T^*_{h-1,d}$ and up to $d-1$ copies of $T^*_{h-1,d}$. The corresponding vertex sets and the graph structure are indicated in the figure. Moreover, although not depicted, as an induced subgraph, $A$ contains all edges of $T^*_{h,d}$ on $V(A)$.}
    \label{fig:admissible}
\end{figure}

\noindent
For an admissible graph $A$ as in Figure~\ref{fig:admissible-b},~\ref{fig:admissible-c} with $h \ge 2$ and a forest $F$ with $|F| < |A|$ a key quantity in the proofs is the \emph{rest} of $F$ that is informally defined as follows. See also Figure~\ref{fig: rest} for an illustration. Let $V\subset V(A)$ be the first $|F|$ vertices in the eating order of $A$, and imagine that we remove $V$ from $A$ to obtain a graph $\widetilde{A}$. Then the rest is the size of the \emph{smallest subgraph rooted at level two} of $\widetilde{A}$ (if such a subgraph exists). If we set $X$ as the size of the smallest and $N$ as the size of the largest subgraph of $A$ rooted at level two, this leads the following definition.
\begin{definition}
\label{def:rest}
    Let $0<X\leq N$.
    For a forest $F$ let the \emph{rest} $\wavytriangle_{N,X}(F)$ be defined by
    $$\wavytriangle_{N,X}(F)=X-|F|\text{ for }|F|< X
    \quad
    \text{and}
    \quad
    \wavytriangle_{N,X}(F)=N-x\text{ for }|F|\geq X,$$
    where, if $|F|\ge X$, then $0\le x<N$ is unique such that $|F|=x+kN+X$ for some $k \ge 0$.
\end{definition}

\begin{figure}
\begin{subfigure}[t]{0.5\textwidth}
    \centering
    \begin{tikzpicture}[scale=0.9, every node/.style={transform shape}]

\node[draw, circle, fill, inner sep=2pt, label=right:{$r_A$}] (root2) at (0,-0.4) {};

\node[draw, circle, fill, inner sep=2pt, label={[label distance=1mm]180:$r_{T^*_{h-1,2}}$}] (child1) at (-1.2,-1) {};
\node[draw, circle, fill, inner sep=2pt, label=right:{$r_{A'}$}] (child2) at (1.2,-1) {};

\node[draw, circle, fill, inner sep=2pt, label={[label distance=1cm]-90:$N$}] (child11) at (-0.6,-1.5) {};
\node[draw, circle, fill, inner sep=2pt, label={[label distance=1cm]-90:$N$}] (child13) at (-1.8,-1.5) {};

\node[draw, black, circle, fill, inner sep=2pt, label={[label distance=1cm]-90:$N$}] (child21) at (0.6,-1.5) {};
\node[draw, black, circle, fill, inner sep=2pt, label={[label distance=0.75cm,color=gray]-90:$X$}] (child22) at (1.8,-1.5) {};
\node[label={[label distance=0.2cm]-90:$Y$}] (child22) at (1.84,-1.5) {};

\draw (root2) -- (child1);
\draw (root2) -- (child2);

\draw (child1) -- (child11);
\draw (child1) -- (child13);

\draw (child2) -- (child21);
\draw (child2) -- (child22);

\draw (-0.6,-1.5) -- (-1.1,-3.15) -- (-0.1,-3.15) -- cycle;
\draw (-1.8,-1.5) -- (-2.3,-3.15) -- (-1.3,-3.15) -- cycle;

\draw (0.6,-1.5) -- (0.1,-3.15) -- (1.1,-3.15) -- cycle;
\draw[draw=gray] (1.8,-1.5) -- (1.3,-3.15) decorate [decoration={zigzag,segment length=2mm}] {-- (2.3,-2.65)} -- cycle;

\draw (1.8,-1.5) -- (1.53,-2.4) decorate [decoration={zigzag,segment length=2mm}] {-- (2.11,-2.21)} -- cycle;

    \end{tikzpicture}
    \caption{$|F|<X$}\label{fig:example_A1}
\end{subfigure}
\hfill
\begin{subfigure}[t]{0.5\textwidth}
    \centering
    \begin{tikzpicture}[scale=0.9, every node/.style={transform shape}]

\node[draw, circle, fill, inner sep=2pt, label=right:{$r_A$}] (root2) at (0,-0.4) {};

\node[draw, circle, fill, inner sep=2pt, label={[label distance=1mm]180:$r_{T^*_{h-1,2}}$}] (child1) at (-1.2,-1) {};
\node[draw, circle, fill, inner sep=2pt, label=right:{$r_{A'}$}] (child2) at (1.2,-1) {};

\node[draw, circle, fill, inner sep=2pt, label={[label distance=1cm]-90:$N$}] (child11) at (-0.6,-1.5) {};
\node[draw, circle, fill, inner sep=2pt, label={[label distance=1cm]-90:$N$}] (child13) at (-1.8,-1.5) {};

\node[draw, black, circle, fill, inner sep=2pt, label={[label distance=1cm,color=gray]-90:$N$}, label={[label distance=0.38cm]-90:$Y$}] (child21) at (0.6,-1.5) {};
\node[draw, black, circle, fill, inner sep=2pt, color=gray, label={[label distance=0.65cm,color=gray]-90:$X$}] (child22) at (1.8,-1.5) {};

\draw (root2) -- (child1);
\draw (root2) -- (child2);

\draw (child1) -- (child11);
\draw (child1) -- (child13);

\draw (child2) -- (child21);
\draw[draw=gray] (child2) -- (child22);

\draw (-0.6,-1.5) -- (-1.1,-3.15) -- (-0.1,-3.15) -- cycle;
\draw (-1.8,-1.5) -- (-2.3,-3.15) -- (-1.3,-3.15) -- cycle;

\draw[draw=gray] (0.6,-1.5) -- (0.1,-3.15) -- (1.1,-3.15) -- cycle;
\draw[draw=gray] (1.8,-1.5) -- (1.3,-3.15) decorate [decoration={zigzag,segment length=2mm}] {-- (2.3,-2.65)} -- cycle;

\draw (0.6,-1.5) -- (0.2,-2.8) decorate [decoration={zigzag,segment length=2mm}] {-- (0.84,-2.3)} -- cycle;

    \end{tikzpicture}
    \caption{$X\leq |F|<N+X$}\label{fig:example_A2}
\end{subfigure}

\vspace{0.5cm}

\begin{subfigure}[t]{0.5\textwidth}
    \centering
    \begin{tikzpicture}[scale=0.9, every node/.style={transform shape}]

\node[draw, circle, fill, inner sep=2pt, label=right:{$r_A$}] (root2) at (0,-0.4) {};

\node[draw, circle, fill, inner sep=2pt, label={[label distance=1mm]180:$r_{T^*_{h-1,2}}$}] (child1) at (-1.2,-1) {};
\node[draw, circle, fill, inner sep=2pt,color=gray,label={[color=gray]right:$r_{A'}$}] (child2) at (1.2,-1) {};

\node[draw, circle, fill, inner sep=2pt, label={[label distance=1cm,color=gray]-90:$N$}, label={[label distance=0.38cm]-90:$Y$}] (child11) at (-0.6,-1.5) {};
\node[draw, circle, fill, inner sep=2pt, label={[label distance=1cm]-90:$N$}] (child13) at (-1.8,-1.5) {};

\node[draw, black, circle, fill, inner sep=2pt,color=gray, label={[label distance=1cm,color=gray]-90:$N$}] (child21) at (0.6,-1.5) {};
\node[draw, black, circle, fill, inner sep=2pt,color=gray, label={[label distance=0.65cm,color=gray]-90:$X$}] (child22) at (1.8,-1.5) {};

\draw (root2) -- (child1);
\draw[draw=gray] (root2) -- (child2);

\draw (child1) -- (child11);
\draw (child1) -- (child13);

\draw[draw=gray] (child2) -- (child21);
\draw[draw=gray] (child2) -- (child22);

\draw[draw=gray] (-0.6,-1.5) -- (-1.1,-3.15) -- (-0.1,-3.15) -- cycle;
\draw (-1.8,-1.5) -- (-2.3,-3.15) -- (-1.3,-3.15) -- cycle;

\draw (-0.6,-1.5) -- (-1,-2.8) decorate [decoration={zigzag,segment length=2mm}] {-- (-0.36,-2.3)} -- cycle;

\draw[draw=gray] (0.6,-1.5) -- (0.1,-3.15) -- (1.1,-3.15) -- cycle;
\draw[draw=gray] (1.8,-1.5) -- (1.3,-3.15) decorate [decoration={zigzag,segment length=2mm}] {-- (2.3,-2.65)} -- cycle;

    \end{tikzpicture}
    \caption{$N+X\leq |F|<2N+X$}\label{fig:example_A3}
\end{subfigure}
\hfill
\begin{subfigure}[t]{0.5\textwidth}
    \centering
    \begin{tikzpicture}[scale=0.9, every node/.style={transform shape}]

\node[draw, circle, fill, inner sep=2pt, label=right:{$r_A$}] (root2) at (0,-0.4) {};

\node[draw, circle, fill, inner sep=2pt, label={[label distance=1mm]180:$r_{T^*_{h-1,2}}$}] (child1) at (-1.2,-1) {};
\node[draw, circle, fill, inner sep=2pt,color=gray,label={[color=gray]right:$r_{A'}$}] (child2) at (1.2,-1) {};

\node[draw, circle, fill, inner sep=2pt,color=gray, label={[label distance=1cm,color=gray]-90:$N$}] (child11) at (-0.6,-1.5) {};
\node[draw, circle, fill, inner sep=2pt, label={[label distance=1cm,color=gray]-90:$N$}, label={[label distance=0.38cm]-90:$Y$}] (child13) at (-1.8,-1.5) {};

\node[draw, black, circle, fill, inner sep=2pt,color=gray, label={[label distance=1cm,color=gray]-90:$N$}] (child21) at (0.6,-1.5) {};
\node[draw, black, circle, fill, inner sep=2pt,color=gray, label={[label distance=0.65cm,color=gray]-90:$X$}] (child22) at (1.8,-1.5) {};

\draw (root2) -- (child1);
\draw[draw=gray] (root2) -- (child2);

\draw[draw=gray] (child1) -- (child11);
\draw (child1) -- (child13);

\draw[draw=gray] (child2) -- (child21);
\draw[draw=gray] (child2) -- (child22);

\draw[draw=gray] (-0.6,-1.5) -- (-1.1,-3.15) -- (-0.1,-3.15) -- cycle;
\draw[draw=gray] (-1.8,-1.5) -- (-2.3,-3.15) -- (-1.3,-3.15) -- cycle;

\draw (-1.8,-1.5) -- (-2.2,-2.8) decorate [decoration={zigzag,segment length=2mm}] {-- (-1.56,-2.3)} -- cycle;

\draw[draw=gray] (0.6,-1.5) -- (0.1,-3.15) -- (1.1,-3.15) -- cycle;
\draw[draw=gray] (1.8,-1.5) -- (1.3,-3.15) decorate [decoration={zigzag,segment length=2mm}] {-- (2.3,-2.65)} -- cycle;

    \end{tikzpicture}
    \caption{$2N+X\leq|F|<3N+X$}\label{fig:example_A4}
\end{subfigure}
    \caption{Let $h\geq 2$. The figure illustrates an admissible graph $A$ for $d=2$, constructed by three copies of $T_{h-2,2}^*$ of size $N$ and one admissible subgraph of $T_{h-2,2}^*$ of size $X$,~cf.~Figure \ref{fig:admissible-b}. The size of the largest (resp.~smallest) subgraph of $A$ rooted at level two is $N$ (resp.~$X$). Let $F$ be any forest with $|F|<3N+X$. 
    After removing the first $|F|$ vertices in the eating order of $A$, the rest $\protect\wavytriangle(F)$ is given by the size of the smallest subgraph rooted at level two, indicated in the figure by $Y$.}
    \label{fig: rest}
\end{figure}

\noindent If $N$ and $X$ are clear from the context, then we abbreviate $\wavytriangle_{N,X}(F)$ by $\wavytriangle(F)$.
In order to use an inductive argument to embed a forest into an admissible graph $A$, it is crucial to describe admissible \emph{sub}graphs of $A$. 
Remark \ref{rem:admissible-recursive} already shows that for any child $c$ of $r_A$, the subgraph rooted at $c$ is also admissible. However, there are several more admissible graphs hidden in $A$.
The first observation directly follows from Definition \ref{def:admissible}. 
\begin{observation}\label{obs:1}
    Let $U$ be the resulting graph after removing the first $0\leq t < |A|$ vertices in the eating order of~$A$. Then $U$ is again admissible.
\end{observation}
Let $D_v$ be empty if $v$ is not in $A$.
Given a vertex $r\in A$, we consider the induced subgraph on $r$, a child $c$ of $r$ with descendants, and up to $d-1$ predecessors of $c$ with descendants.
The second observation is that the (recursive) structure of $A$ from Remark~\ref{rem:admissible-recursive} implies that this subgraph is admissible. 
\begin{observation}\label{obs:2} 
    Let $r\in A$ and $c$ be a child of $r$. Let $1\leq t \leq d-1$ be such that $c-t$ is lexicographically larger than $c$ and set $D=\bigcup_{0\leq i\leq t}(\{c-i\} \cup D_{c-i})$. 
    Then the induced subgraph of $A$ with root $r$ and descendants $D$ is admissible.
\end{observation}
As we will see, these (simple) observations allow us to prove that any admissible graph $A$ is universal for all forests $F$ with at most $|A|$ vertices, and consequently for all trees of that size. Furthermore, when $|F|<|A|$ we not only establish that $A$ contains a subgraph $U$ isomorphic to $F$, we even ensure that there exists an \emph{embedding} of $F$ into $A$, which (slightly abusing notation) additionally requires $U$ to be given by the first $|F|$ vertices in the eating order of $A$. 
\begin{definition}
    For any forest $F$ and admissible graph $A$, a mapping $\lambda:V(F) \to V(A)$ is an \emph{embedding} of $F$ into $A$ if 
\begin{itemize}
    \item $\lambda(V(F))$ consists of the first $|F|$ vertices in the eating order of $A$, and
    \item for every edge $uv$ in $F$, the vertices $\lambda(u)$ and $\lambda(v)$ are adjacent in $A$.
\end{itemize}
\end{definition}
Note that if there exists an embedding $\lambda$ of $F$ into $A$, then by Observation \ref{obs:1} the graph $A\setminus\lambda(V(F))$ is admissible.
With all these definitions and properties at hand we can now state a central ingredient of our proof. It establishes that all admissible graphs $A$ are universal for forests~$F$ with $|F|<|A|$ in the stronger sense that there is always an embedding of $F$ into $A$. 

\begin{lemma}\label{lem:embedding} 
    Let $d\in\{2,3\}$, $A$ be an admissible graph and $F$ a forest with $|F| < |A|$. Then there exists an embedding $\lambda$ of $F$ into $A$. In particular, $A\setminus \lambda(V(F))$ is admissible.
\end{lemma}
With this at hand, Lemma \ref{lem:univ}, which asserts that $U_{n,d}$ is universal for $d\in\{2,3\}$, follows immediately from the next remark.
\begin{remark}\label{rem:univ}
    Let $d\in\{2,3\}$ and $A$ be any admissible graph with $n$ vertices, in particular $A$ could be some $U_{n,d}$. For any tree $T$ with $n$ vertices and any vertex $v$ in $T$, we can embed $T\setminus v$ into $A$ using Lemma \ref{lem:embedding} and then place $v$ at the root $r_A$, which is connected to \emph{all} of its descendants thanks to the \ref{en: type 1 edges} edges. This shows that $A$ is universal.
\end{remark}
The second main ingredient in our proof is the following lemma that counts the edges in $U_{n,d}$. Together with the aforementioned remark, this concludes the proof of Theorem~\ref{thm:main}.
\begin{lemma}\label{lem:size}
If $d\neq 3$, then $U_{n,d}$ has $\frac{d}{\ln d}n \ln n + O(n)$ edges. Moreover, $U_{n,3}$ has $\frac{19}{6\ln3}n \ln n + O(n)$ edges.
\end{lemma}
The proof is given in Section~\ref{sec:size}. Note that the \ref{en: type 3 edges} edges explain the exceptional result for $d=3$. Observe that by combining this lemma together with Lemma~\ref{lem:univ} we directly obtain Theorem \ref{thm:main}. However, some remarks are in place.
First of all,
\[
    \frac{2}{\ln 2} = 2.88539 \dots
    \quad
    \text{and}
    \quad
    \frac{19}{6\ln 3} = 2.88242 \dots
\]
so that the result for $d=3$ is only \emph{very slightly} better than that for $d=2$. The proof, however, is more complex, since it leverages an additional idea: instead of splitting a tree by removing just one vertex, we split it into several parts, whose sizes we can control better, by removing two vertices. The ternary structure can accommodate such maneuvers, so that in the end we do get an improved result. In any case, we also decided to include the proof for the case when $d=2$, as it is illustrative and guides through the more involved arguments for $d=3$.
Finally, to close this section note that since $ \frac{19}{6\ln3} < \frac{d}{\ln d}$ for $d \neq 3$, it makes no sense to study universality properties of $U_{n,d}$ for $d \neq 3$, and we explicitly do not do that. 

\paragraph{Outline} The remainder of the paper is organized as follows. In Section~\ref{sec: binary trees} we show the case $d=2$ of Lemma~\ref{lem:embedding}, thus establishing universality of $U_{n,2}$. Then, we establish the case $d=3$ of Lemma~\ref{lem:embedding} in Section~\ref{sec:ternary trees}, showing that $U_{n,3}$ is universal. The proofs of the auxiliary Lemma~\ref{lem:sep-main} and other lemmas about separating vertices in trees are presented in Section~\ref{sec:other proofs}. Finally, in Section~\ref{sec: results treewidth}, we extend our results to graphs with treewidth $w$, proving Theorem~\ref{thm:main-treewidth} and presenting the blown-up universal graphs.

\section{Proof of Lemma \ref{lem:embedding} for $d=2$}\label{sec: binary trees} 
We prove Lemma \ref{lem:embedding} for $d=2$ by induction over $|A|\in\mathbb{N}$.
Note that $A$ is a complete graph for $h<2$. We thus assume that $h\ge 2$, so that we are in the setting of Figure~\ref{fig:admissible}.
For the induction step, let $F$ be a forest with $0<|F|<|A|$, let $V\subset V(A)$ be the first $|F|$ vertices in the eating order for $A$ and assume that the statement holds for all admissible graphs $A'$ and forests $F'$ with $|F'|<|A'|<|A|$.

First, consider the case that $A$ is of the type shown in Figure~\ref{fig:admissible-a}, so $A$ consists of a root $r_A$ and an admissible subgraph $A'$. Since $|F|<|A|$ and $r_A$ is last in the eating order of $A$, we obtain $V\subseteq A'$. Moreover, the eating order on $V$ is the same in $A$ and $A'$.
If $|F| < |A'|$, we apply the induction hypothesis to $F$ and $A'$, which proves the claim.
Otherwise $|F| = |A'|$. Let $v\in V(F)$ arbitrary. We embed $F \setminus v$ into $A'$ using the induction hypothesis and then place $v$ at $r_{A'}$. This yields a proper embedding, as $r_{A'}$ is connected to every vertex in $A'$ due to the \ref{en: type 1 edges} edges and $F$ is embedded at $V$.

\begin{figure}[b]
\begin{subfigure}[t]{0.5\textwidth} 
\centering
\begin{tikzpicture}[scale=0.85, every node/.style={transform shape}]
    \node[draw, circle, fill=black, inner sep=2pt, label=right:{$r_A$}] (root2) at (0,-0.2) {};

    \node[draw, circle, fill=black, inner sep=2pt, label={[label distance=1mm]180:$r_{T^*_{h-1,2}}$}] (child1) at (-1.75,-1) {};
    \node[draw, circle, fill=black, inner sep=2pt, label=right:{$r_{A'}$}] (child2) at (1.75,-1) {};

    \node[draw, circle, fill=black, inner sep=2pt, label={[label distance=0.8cm]-90:$N$}] (child11) at (-1,-1.5) {};
    \node[draw, circle, fill=black, inner sep=2pt, label={[label distance=0.8cm]-90:$N$}] (child13) at (-2.5,-1.5) {};

    \node[draw, circle, fill=black, inner sep=2pt, label={[label distance=0.45cm]-90:$X$}] (child21) at (1,-1.5) {};
    
    \draw (root2) -- (child1);
    \draw (root2) -- (child2);
    \draw (child1) -- (child11);
    \draw (child1) -- (child13);
    \draw (child2) -- (child21);

    \draw[dashed] (-1.9,-3.2) -- (2.7,-3.2) -- (2.7,-0.6) -- (-1.2,-1) -- cycle;
    \node[label=above:{$U$}] (label) at (2.2,-0.7) {};
    \draw[dashed] (child2) -- (child11);
        
    \draw (-1,-1.5) -- (-1.5,-3) -- (-0.5,-3) -- cycle; 
    \draw (-2.5,-1.5) -- (-3,-3) -- (-2,-3) -- cycle; 

    \draw (1,-1.5) -- (0.5,-3) decorate [decoration={zigzag,segment length=2mm}] {-- (1.5,-2.5)} -- cycle; 
       
\end{tikzpicture}
\caption{}\label{fig:IH1_a}
\end{subfigure}
\hfill
\begin{subfigure}[t]{0.5\textwidth} 
\centering
\begin{tikzpicture}[scale=0.85, every node/.style={transform shape}]
    \node[draw, circle, fill=black, inner sep=2pt, label=right:{$r_A$}] (root2) at (0,-0.2) {};

    \node[draw, circle, fill=black, inner sep=2pt, label={[label distance=1mm]180:$r_{T^*_{h-1,2}}$}] (child1) at (-1.75,-1) {};
    \node[draw, circle, fill=black, inner sep=2pt, label=right:{$r_{A'}$}] (child2) at (1.75,-1) {};

    \node[draw, circle, fill=black, inner sep=2pt, label={[label distance=0.8cm]-90:$N$}] (child11) at (-1,-1.5) {};
    \node[draw, circle, fill=black, inner sep=2pt, label={[label distance=0.8cm]-90:$N$}] (child13) at (-2.5,-1.5) {};

    \node[draw, circle, fill=black, inner sep=2pt, label={[label distance=0.8cm]-90:$N$}] (child21) at (1,-1.5) {};
    \node[draw, circle, fill=black, inner sep=2pt, label={[label distance=0.45cm]-90:$X$}] (child22) at (2.5,-1.5) {};
    
    \draw (root2) -- (child1);
    \draw (root2) -- (child2);
    \draw (child1) -- (child11);
    \draw (child1) -- (child13);
    \draw (child2) -- (child21);
    \draw (child2) -- (child22);

    \draw[dashed] (0.2,-3.2) -- (3.3,-3.2) -- (3.3,-1.5) -- (2.1,-0.5) -- (1.35,-0.5) -- (0.2,-1.5) -- cycle;
    \node[label=above:{$U$}] (label) at (2.4,-0.7) {};
        
    \draw (-1,-1.5) -- (-1.5,-3) -- (-0.5,-3) -- cycle; 
    \draw (-2.5,-1.5) -- (-3,-3) -- (-2,-3) -- cycle; 

    \draw (1,-1.5) -- (0.5,-3) -- (1.5,-3) -- cycle; 
    \draw (2.5,-1.5) -- (2.0,-3) decorate [decoration={zigzag,segment length=2mm}] {-- (3,-2.5)} -- cycle; 
\end{tikzpicture}
\caption{}\label{fig:IH1_b}
\end{subfigure}
\caption{
This figure illustrates the two possibilities of an admissible graph $A$ of type as in Figure~\ref{fig:admissible-b} and $|A'|\geq 2$. Let $U$ be the induced subgraph containing the first $N+X+1$ vertices in the eating order of $A$. If $r_{A'}$ only has one child, $U$ is indicated by the dashed region in (a), otherwise in (b). By Observation~\ref{obs:2}, $U$ is admissible in both cases. 
}
\label{fig:IH1}
\end{figure}


Since $d=2$ we are left with the case that $A$ is as in Figure~\ref{fig:admissible-b}, so that it consists of a root $r_A$, one subgraph $T^*_{h-1,2}$ and another admissible graph $A'$, where $A'$ is first in the eating order.
If $|A'|=1$, we place any $v\in V(F)$ at $r_{A'}$, which is first in the eating order. By Observation \ref{obs:1}, the graph $A\setminus r_{A'}$ is again admissible.
Thus, we use the induction hypothesis to embed $F\setminus v$ into $A\setminus r_{A'}$ such that $F\setminus v$ is placed at the first $|F\setminus v|$ vertices in the eating order of $A\setminus r_{A'}$. Therefore, $F$ is embedded exactly at the vertices $V$. Moreover, the \ref{en: type 2 edges} edges ensure that $r_{A'}$ is connected to every vertex in $T^*_{h-1,2}$, thereby defining an embedding.

In the rest we assume that we are in the situation of Figure~\ref{fig:admissible-b} and $|A'| \ge 2$.  Let $X:=|A''|$, where~$A''$ is the subgraph rooted at the lexicographically largest vertex among the children of $r_{A'}$ in~$A$, see also Figure~\ref{fig:IH1}. Moreover, set $N := |T^*_{h-2,2}|>0$.
If $|F| \leq N+X+1$, we consider the induced subgraph $U$ on $r_{A'}$, the vertices in $A''$, and on $u=r_{A''}-1$ with its descendants $D_u$. 
If $r_{A'}$ only has one child, $U$ is shown in Figure \ref{fig:IH1_a}, else $r_{A'}$ has two children depicted in Figure \ref{fig:IH1_b}.
In both cases,~$U$ contains the vertices in~$V$ and is admissible by Observation~\ref{obs:2}.
If the root $r_U=r_{A'}$ is not in $V$, we embed $F$ into~$U$ using the induction hypothesis. Thus, $F$ is embedded exactly at~$V$, proving the claim. Otherwise, if~$r_U$ is in~$V$, let $v$ be an arbitrary vertex of $F$. We use the induction hypothesis to embed $F\setminus v$ into $U$. This embeds $F\setminus v$ at $V\setminus r_U$. Finally, we place $v$ at $r_U$, such that $F$ eats $V$. The~\ref{en: type 1 edges} and~\ref{en: type 2 edges} edges ensure that $r_U$ is connected to every vertex in $U$, proving the claim also in that case. 

It remains to treat the case where we are in the situation of Figure~\ref{fig:admissible-b}, $|A'| \ge 2$, and $|F| > N+X+1$. First, assume that $|F|<3N+X+2$.
Let
\begin{equation}
\label{eq:mM}
    m:=\max\{0,|F|-2N-X-1\}
    \quad
    \text{and}
    \quad
    M:=|F|-N-X-1.
\end{equation}
Note that $2m\leq M$. Applying Lemma \ref{lem:sep-main} to $F$ gives us a separating vertex $s$ and a partition of~$F\setminus s$ in three forests such that
\begin{align}\label{partition}
m \leq |F_3| \leq M,\quad|F_1| \leq |F| -1 -M,\quad\text{and}\quad|F_2| \leq |F_1|.
\end{align}
As we will argue, these conditions ensure that $F_1$, $F_2$, $F_3$ have appropriate sizes so that they can be embedded one after another into $A$ by using the induction hypothesis.
First, since $|F_1|\leq N+X$, we can embed $F_1$ as in the previous case using $U$, depicted in Figure \ref{fig:IH1}, leaving an admissible graph $A^{(1)}$ with $r_A = r_{A^{(1)}}$. Moreover, this embedding is such that $r_{A'}$ is not used, since $|F_1| < |U|$ and the root is eaten last.
By \eqref{partition}, we obtain $|F_1|\geq X$, since
\[
2|F_1|\geq |F_1|+|F_2| = |F|-1-|F_3|\geq N+X \geq 2X.
\]

\begin{figure}[b]
\begin{subfigure}[t]{0.3\textwidth} 
\centering
\begin{tikzpicture}[scale=0.85, every node/.style={transform shape}]
    \node[] at (0,1) {};
    \node[draw, circle, fill=black, inner sep=2pt, label=right:{$r_{A^{(1)}}$}] (root2) at (0,-0.2) {};

    \node[draw, circle, fill=black, inner sep=2pt, label={[label distance=1mm]180:$r_{T^*_{h-1,2}}$}] (child1) at (-1.6,-1) {};
    \node[draw, circle, fill=black, inner sep=2pt, label=right:{$r_{A'}$}] (child2) at (1.6,-1) {};

    \node[draw, circle, fill=black, inner sep=2pt, label={[label distance=0.8cm]-90:$N$}] (child13) at (-2.35,-1.5) {};

    \draw (root2) -- (child1);
    \draw (root2) -- (child2);
    \draw (child1) -- (child13);

    \draw[dashed] (1.1,-3.2) -- (-3.2,-3.2) -- (-3.2,-0.8) -- (-1,0.2) -- (1.1,0.2) --  cycle;
    \node[label=above:{$U^{(1)}$}] (label) at (-1.7,-0.1) {};

    \draw (-2.35,-1.5) -- (-2.85,-3) -- (-1.85,-3) -- cycle; 

\end{tikzpicture}
\caption{}\label{fig:IH2_a}
\end{subfigure}
\hfill
\begin{subfigure}[t]{0.3\textwidth} 
\centering
\begin{tikzpicture}[scale=0.85, every node/.style={transform shape}]
    \node[draw, circle, fill=black, inner sep=2pt, label=right:{$r_{A^{(1)}}$}] (root2) at (0,-0.2) {};

    \node[draw, circle, fill=black, inner sep=2pt, label={[label distance=1mm]180:$r_{T^*_{h-1,2}}$}] (child1) at (-1.6,-1) {};
    \node[draw, circle, fill=black, inner sep=2pt, label=right:{$r_{A'}$}] (child2) at (1.6,-1) {};

    \node[draw, circle, fill=black, inner sep=2pt, label={[label distance=0.45cm]-90:$Y$}] (child11) at (-0.85,-1.5) {};
    \node[draw, circle, fill=black, inner sep=2pt, label={[label distance=0.8cm]-90:$N$}] (child13) at (-2.35,-1.5) {};

    \draw (root2) -- (child1);
    \draw (root2) -- (child2);
    \draw (child1) -- (child11);
    \draw (child1) -- (child13);

    \draw[dashed] (1.1,-3.2) -- (-3.2,-3.2) -- (-3.2,-0.8) -- (-1,0.2) -- (1.1,0.2) --  cycle;
    \node[label=above:{$U^{(1)}$}] (label) at (-1.7,-0.1) {};
        
    \draw (-2.35,-1.5) -- (-2.85,-3) -- (-1.85,-3) -- cycle; 

    \draw (-0.85,-1.5) -- (-1.35,-3) decorate [decoration={zigzag,segment length=2mm}] {-- (-0.35,-2.5)} -- cycle; 

\end{tikzpicture}
\caption{}\label{fig:IH2_b}
\end{subfigure}
\hfill
\begin{subfigure}[t]{0.3\textwidth} 
\centering
\begin{tikzpicture}[scale=0.85, every node/.style={transform shape}]
    \node[draw, circle, fill=black, inner sep=2pt, label=right:{$r_{A^{(1)}}$}] (root2) at (0,-0.2) {};

    \node[draw, circle, fill=black, inner sep=2pt, label={[label distance=1mm]180:$r_{T^*_{h-1,2}}$}] (child1) at (-1.6,-1) {};
    \node[draw, circle, fill=black, inner sep=2pt, label=right:{$r_{A'}$}] (child2) at (1.6,-1) {};

    \node[draw, circle, fill=black, inner sep=2pt, label={[label distance=0.8cm]-90:$N$}] (child11) at (-0.85,-1.5) {};
    \node[draw, circle, fill=black, inner sep=2pt, label={[label distance=0.8cm]-90:$N$}] (child13) at (-2.35,-1.5) {};

    \node[draw, circle, fill=black, inner sep=2pt, label={[label distance=0.45cm]-90:$Y$}] (child21) at (0.85,-1.5) {};
    
    \draw (root2) -- (child1);
    \draw (root2) -- (child2);
    \draw (child1) -- (child11);
    \draw (child1) -- (child13);
    \draw (child2) -- (child21);

    \draw[dashed] (-1.75,-3.2) -- (2.5,-3.2) -- (2.5,-0.6) -- (-1.1,-1) -- cycle;
    \node[label=above:{$U^{(1)}$}] (label) at (2.1,-0.7) {};
    \draw[dashed] (child2) -- (child11);
        
    \draw (-0.85,-1.5) -- (-1.35,-3) -- (-0.35,-3) -- cycle; 
    \draw (-2.35,-1.5) -- (-2.85,-3) -- (-1.85,-3) -- cycle; 

    \draw (0.85,-1.5) -- (0.35,-3) decorate [decoration={zigzag,segment length=2mm}] {-- (1.35,-2.5)} -- cycle; 
     
\end{tikzpicture}
\caption{}\label{fig:IH2_c}
\end{subfigure}
\caption{
This figure illustrates the three possible shapes of $A^{(1)}$, the graph obtained after embedding a forest $F_1$ with $X\leq |F_1|\leq N+X$ and rest $Y=\protect\wavytriangle(F_1)$ into the admissible graph $A$ in Figure \ref{fig:IH1}. 
By Observations~\ref{obs:1} and~\ref{obs:2}, in all cases the depicted subgraph on $U^{(1)}$ is admissible.
}
\label{fig:IH2}
\end{figure}

\noindent 
Thus, there are exactly three possible cases for the shape of~$A^{(1)}$, shown in Figure \ref{fig:IH2}. 
If~$r_{A'}$ does not have a child in~$A^{(1)}$, i.e., as in ~Figure \ref{fig:IH2_a} or \ref{fig:IH2_b}, we place~$s$ at~$r_{A'}$. Thus, $F_1\cup\{s\}$ is embedded at the first $|F_1|+1$ vertices of $A$ and the remaining graph $U^{(1)}$ is admissible.
We apply the induction hypothesis to embed~$F_2 \cup F_3$ into $U^{(1)}$, so that~$F$ eats~$V$.
This yields a proper embedding, since $r_{A'}$ is connected to every vertex in $T^*_{h-1,2}$ thanks to the \ref{en: type 2 edges} edges.
Otherwise, in the setting of Figure~\ref{fig:IH2_c}, observe that \eqref{partition} implies $|F_1|+|F_2|\leq 2N+X$, and so
\[
|F_2|\leq 2N+X-|F_1| \leq N+\wavytriangle(F_1).
\]
Hence, we embed~$F_2$ similar to~$F_1$ using the admissible subgraph~$U^{(1)}$ in Figure~\ref{fig:IH2_c} and place~$s$ at~$r_{A'}$.
By \eqref{partition}, $|F_1|+|F_2|\geq N+X$, such that the embedding of~$F_1 \cup F_2 \cup \{s\}$ eats the first~$|F_1|+|F_2|+1$ vertices in~$A$ and the remaining graph~$A^{(2)}$ is again admissible. If~$F_3 =\emptyset$, this proves the claim, else we embed~$F_3$ into~$A^{(2)}$ by the induction hypothesis. 
Therefore,~$F$ eats exactly~$V$ again.

The last case to consider is when~$|F|=3N+X+2$. Here, we choose~$m=N$ -- opposed to~$N+1$ in~\eqref{eq:mM} -- so that~$M=2N+1\ge 2m$. This allows $|F_3|=N$, which creates the problem that $F_2$ might spread over three subgraphs rooted at level two,~i.e.~$|F_2|=N+\wavytriangle(F_1)+1$.
Therefore, we require~$|F_3|$ to be maximal in \eqref{partition}. This implies $|F_3|>N$, because otherwise
\[
m=|F_3| < |F_2|\leq |F_1|\leq N+X\leq M
\quad 
\text{and}
\quad 
|F_3|\leq |F_1|
\]
such that changing the roles of $F_2$ and $F_3$ contradicts the maximality of $|F_3|$.
Hence, $|F_3|>N$, which implies $|F_1|+|F_2|\leq 2N+X$ and consequently $|F_2|\leq N +\wavytriangle(F_1)$. With this at hand, the remainder follows exactly as above.

\section{Proof of Lemma \ref{lem:embedding} for $d=3$}
\label{sec:ternary trees}

In this section $d=3$ is fixed, hence we write $T_{h}^*$ for $T_{h,3}^*$.
In the proof, where we want to construct an embedding of $F$ into $A$, we proceed by induction over $|A|\in \mathbb{N}$.
As in the case $d=2$ in the previous section, we will separate $F$ into smaller parts and identify appropriate subgraphs of $A$, where we can use the induction hypothesis to embed them.
Crucially, for the hypothesis to apply, the subgraphs of $A$ that we will consider must be admissible, and this will necessarily impose size constraints on the parts of $F$.
In Section~\ref{sec: binary trees}, applying Lemma~\ref{lem:sep-main} directly allowed us to handle these cases by using one separating vertex. Here, one separating will generally be not sufficient due to the more involved $3$-ary base structure. In the following subsection we suitably adapt Lemma~\ref{lem:sep-main} so that it works in the present case as well, and we explain what this exactly means. The proof of Lemma~\ref{lem:embedding} is then given in Subsection~\ref{subsec:proof}.

\subsection{Splitting Forests}
\label{subsec:preliminaries}

In order to simplify the exposition, in this section we assume that $A$ is fixed and of type as shown in Figure~\ref{fig:admissible-b} or~\ref{fig:admissible-c} with $|A'|\geq 2$. (The other cases will turn out to be simple in the proof of Lemma \ref{lem:embedding}.) We denote by $N$ (resp.~$X$) the size of the largest (resp.~smallest) subgraph of $A$ rooted at level two of $A$.
Let $F$ be the graph that we want to embed into $A$.
The idea is to remove \emph{up to two} vertices from $F$ such that we can apply the induction hypothesis. In fact, for~$d=3$ and by Observation~\ref{obs:2}, the induction hypothesis applies in particular to subgraphs of $A$ not extending over more than \emph{three} -- as opposed to \emph{two} in Section \ref{sec: binary trees} -- subgraphs rooted at level two of $A$, see also Figure~\ref{fig:IH1 ternary}. To be precise, this means that if $|F|\leq 2N+X$, then we can embed it by using the induction hypothesis, and hence we call $2N+X$ the ``magic size'' with respect to $A$.
Our general aim in this section is to show that any $F$ whose size is larger than the magic size can be partitioned, by removing one or two vertices, into forests $F_1, F_2, \dots$ whose size is at most the magic size.
Note that this means that $|F_1| \le 2N + X$, but then, after embedding $F_1$, the new magic size is $2N + \wavytriangle(F_1)$, where $\wavytriangle(F_1)$ is the rest that remains from a full subgraph at level two after embedding $F_1$, see Definition~\ref{def:rest}.
So, we require that $|F_2| \le 2N + \wavytriangle(F_1)$, and then, iteratively $|F_3| \le 2N + \wavytriangle(F_1 \cup F_2)$, and so on.

The first case that we consider is when $|F|\leq 5N+X+2$, where the next lemma shows that one separating vertex is actually enough to (essentially) fulfill the magic size constraints.   
\begin{lemma}\label{lem:sep-main-cons1}
Let $X > 0$ and $N\geq X$. Let $F$ be a forest with $2N+X+2\leq |F|\leq 5N+X+2$. Then there exists a vertex $s \in F$ and a partition $F_1,F_2,F_3$ of $F\setminus s$ such that 
\[
|F_1| \leq 2N+X,
\quad
|F_2| \leq 2N + \wavytriangle(F_1),
\quad
|F_3| \leq 2N + \wavytriangle(F_1 \cup F_2)+1,
\quad
\text{and}
\quad
|F_1| +|F_2|\geq 2N+X.
\]    
\end{lemma}
The proof is in Section \ref{subsec:proof of sep-main-cons1}.
Note that the fourth conclusion in the lemma ensures that $F_1$ and $F_2$ cover the first three subgraphs rooted at level two in the eating order of $A$, i.e., the subgraph $U$ in Figure~\ref{fig:IH1 ternary}, and this will be very helpful when embedding $F$. This is the main difference to the other  case $|F|\geq 5N+X+3$, where we are in the setting of Figure \ref{fig:admissible-c} and where we will need to cut twice.
\begin{lemma}\label{lem:sep-main-cons2}
Let $X > 0$ and $N\geq X$. Let $F$ be a forest with $5N+X+3\leq|F|\leq 8N+X+3$.
Let $I_i\in\{0,1\}$ be $1$ if $|F|=8N+X+3$ and $i=6$.
Then there exists a vertex $s_1\in F$ and a partition $F_1,F_2,F_4,\overline F$
of $F\setminus s_1$, a vertex $s_2\in\overline F$ and a partition $F_3,F_5,F_6$ of $\overline F\setminus s_2$ such that 
\[
    |F_i| \leq 2N + \wavytriangle\left(\bigcup_{1\leq j <i} F_j \right) + I_i\quad\text{for }i\geq 1,
    \quad\text{and}\quad
    |F_1|+|F_2|\ge\frac32N+X.
\]
\end{lemma}
The proof is in Section~\ref{subsec:proof of sep-main-cons2}.
By the first condition, all forests satisfy the magic size constraint (except for $F_6$, which won't be a problem). Additionally, the second condition ensures that $F_1$ and $F_2$ eat at least $\frac32N+X$ vertices that will impose additional restrictions in the proof of Lemma \ref{lem:embedding}. In particular, as we shall see, this is the ultimate reason why we have -- in contrast to the case $d=2$ -- to include the \ref{en: type 3 edges} edges in our construction.

\subsection{Main Proof of Lemma \ref{lem:embedding}}\label{subsec:proof}
The proof follows by induction over $|A|\in \mathbb{N}$.
For~$h<2$, every admissible graph is complete. Hence, let~$h\geq 2$ and assume that the statement holds for all admissible graphs~$A'$ and forests~$F'$ with~$|F'|<|A'|<|A|$. Let~$F$ be a forest with~$0<|F|<|A|$, let~$V\subset V(A)$ be the first~$|F|$ vertices in the eating order of~$A$. We distinguish three cases according to the shape of~$A$ given by Figure~\ref{fig:admissible}.

First, assume that~$A$ is of type shown in Figure~\ref{fig:admissible-a}, consisting of a root~$r_A$ and an admissible subgraph~$A'$. If~$|F|<|A'|$, we directly apply the induction hypothesis to embed $F$ into $A'$. Since $|F|<|A|$, the remaining case is $|F|=|A'|$. We choose an arbitrary vertex $v\in F$, place it at $r_{A'}$ and embed $F\setminus v$ into $A'$ using the induction hypothesis. Due to the \ref{en: type 1 edges} edges, $r_{A'}$ is connected to every vertex in $A'$, which proves the claim.

Next, assume that~$A$ has shape as in Figure~\ref{fig:admissible-b}, so~$A$ consists of its root~$r_A$, one copy~$T_{h-1}^2$ of~$T_{h-1}^*$ and another admissible subgraph~$A'$, where $A'$ is first in the eating order.
If $|A'|=1$, we place any vertex $v\in F$ at $r_{A'}$, which is first in the eating order and connected to every vertex in~$T_{h-1}^2$ thanks to the~\ref{en: type 2 edges} edges.
By Observation \ref{obs:1}, the graph $A\setminus r_{A'}$ is admissible and we apply the induction hypothesis to $A\setminus r_{A'}$ and $F\setminus v$. Thus, $F$ is embedded at $V$, proving the claim.

\begin{figure}[b!]
\centering
\begin{subfigure}[t]{0.45\textwidth}
\centering
\begin{tikzpicture}[scale=0.85, every node/.style={transform shape}]
    \node[draw, circle, fill=black, inner sep=2pt, label={[label distance=-1mm]-280:{$r_A$}}] (root2) at (0,-0.2) {};

    \node[draw, circle, fill=black, inner sep=2pt, label={[label distance=1mm]180:$r_{T^3_{h-1}}$}] (child1) at (-2.5,-1) {};
    \node[draw, circle, fill=black, inner sep=2pt, label={[label distance=-3mm]-280:$r_{T^2_{h-1}}$}] (child2) at (0,-1) {};
    \node[draw, circle, fill=black, inner sep=2pt, label=right:{$r_{A'}$}] (child3) at (3.3,-1) {};

    \node[draw, circle, fill=black, inner sep=2pt, label={[label distance=0.8cm]-90:$N$}] (child21) at (-1.1,-1.5) {};
    \node[draw, circle, fill=black, inner sep=2pt, label={[label distance=0.8cm]-90:$N$}] (child22) at (0,-1.5) {};
    \node[draw, circle, fill=black, inner sep=2pt, label={[label distance=0.8cm]-90:$N$}] (child23) at (1.1,-1.5) {};

    \node[label={[label distance=0.8cm]-90:$|T^*_{h-1}|$}] (label1) at (-2.5,-1.5) {};

    \node[draw, circle, fill=black, inner sep=2pt, label={[label distance=0.45cm]-90:$X$}] (child31) at (2.2,-1.5) {};
    
    \draw (root2) -- (child1);
    \draw (root2) -- (child2);
    \draw (root2) -- (child3);

    \draw (child2) -- (child21);
    \draw (child2) -- (child22);
    \draw (child2) -- (child23);

    \draw (child3) -- (child31);

    \draw[dashed] (-0.55,-3.2) -- (4.2,-3.2) -- (4.2,-0.65) -- (-0.2,-1.25) -- cycle;
    \node[label=above:{$U$}] (label) at (4.1,-0.75) {};
    \draw[dotted] (-1.6,-3.4) -- (4.5,-3.4) -- (4.5,0.3) -- (-1.6,0.3) -- cycle;
    \draw[dashed] (child3) -- (child22);
    \draw[dashed] (child3) -- (child23);
        
    \draw (-2.5,-1) -- (-3.2,-3) -- (-1.8,-3) -- cycle; 

    \draw (-1.1,-1.5) -- (-1.5,-3) -- (-0.7,-3) -- cycle; 
    \draw (0,-1.5) -- (-0.4,-3) -- (0.4,-3) -- cycle; 
    \draw (1.1,-1.5) -- (0.7,-3) -- (1.5,-3) -- cycle; 

    \draw (2.2,-1.5) -- (2.6,-2.5) decorate [decoration={zigzag,segment length=2mm}] {-- (1.75,-3)} -- cycle; 
    
\end{tikzpicture}
\caption{}\label{fig:IH1_a ternary}
\end{subfigure}
\hfill
\begin{subfigure}[t]{0.45\textwidth} 
\centering
\begin{tikzpicture}[scale=0.85, every node/.style={transform shape}]
    \node[draw, circle, fill=black, inner sep=2pt, label={[label distance=-1mm]-280:{$r_A$}}] (root2) at (0,-0.2) {};

    \node[draw, circle, fill=black, inner sep=2pt, label={[label distance=1mm]180:$r_{T^3_{h-1}}$}] (child1) at (-2.5,-1) {};
    \node[draw, circle, fill=black, inner sep=2pt, label={[label distance=-3mm]-280:$r_{T^2_{h-1}}$}] (child2) at (0,-1) {};
    \node[draw, circle, fill=black, inner sep=2pt, label=right:{$r_{A'}$}] (child3) at (3.3,-1) {};

    \node[draw, circle, fill=black, inner sep=2pt, label={[label distance=0.8cm]-90:$N$}] (child21) at (-1.1,-1.5) {};
    \node[draw, circle, fill=black, inner sep=2pt, label={[label distance=0.8cm]-90:$N$}] (child22) at (0,-1.5) {};
    \node[draw, circle, fill=black, inner sep=2pt, label={[label distance=0.8cm]-90:$N$}] (child23) at (1.1,-1.5) {};

    \node[label={[label distance=0.8cm]-90:$|T^*_{h-1}|$}] (label1) at (-2.5,-1.5) {};

    \node[draw, circle, fill=black, inner sep=2pt, label={[label distance=0.8cm]-90:$N$}] (child31) at (2.2,-1.5) {};
    \node[draw, circle, fill=black, inner sep=2pt, label={[label distance=0.45cm]-90:$X$}] (child32) at (3.3,-1.5) {};
    
    \draw (root2) -- (child1);
    \draw (root2) -- (child2);
    \draw (root2) -- (child3);

    \draw (child2) -- (child21);
    \draw (child2) -- (child22);
    \draw (child2) -- (child23);

    \draw (child3) -- (child31);
    \draw (child3) -- (child32);

    \draw[dashed] (0.5,-3.2) -- (4.2,-3.2) -- (4.2,-0.6) -- (0.85,-1.25) -- cycle;
    \node[label=above:{$U$}] (label) at (4.1,-0.7) {};
    \draw[dotted] (-1.6,-3.4) -- (4.5,-3.4) -- (4.5,0.3) -- (-1.6,0.3) -- cycle;
    \draw[dashed] (child3) -- (child23);
        
    \draw (-2.5,-1) -- (-3.2,-3) -- (-1.8,-3) -- cycle; 

    \draw (-1.1,-1.5) -- (-1.5,-3) -- (-0.7,-3) -- cycle; 
    \draw (0,-1.5) -- (-0.4,-3) -- (0.4,-3) -- cycle; 
    \draw (1.1,-1.5) -- (0.7,-3) -- (1.5,-3) -- cycle; 

    \draw (2.2,-1.5) -- (1.8,-3) -- (2.6,-3) -- cycle; 
    \draw (3.3,-1.5) -- (3.7,-2.5) decorate [decoration={zigzag,segment length=2mm}] {-- (2.85,-3)} -- cycle; 
    
\end{tikzpicture}
\caption{}\label{fig:IH1_b ternary}
\end{subfigure}
\begin{subfigure}[t]{0.5\textwidth} 
\centering
\begin{tikzpicture}[scale=0.85, every node/.style={transform shape}]
    \node[draw, circle, fill=black, inner sep=2pt, label={[label distance=-1mm]-280:{$r_A$}}] (root2) at (0,-0.2) {};

    \node[draw, circle, fill=black, inner sep=2pt, label={[label distance=1mm]180:$r_{T^3_{h-1}}$}] (child1) at (-2,-1) {};
    \node[draw, circle, fill=black, inner sep=2pt, label={[label distance=-3mm]-280:$r_{T^2_{h-1}}$}] (child2) at (0,-1) {};
    \node[draw, circle, fill=black, inner sep=2pt, label=right:{$r_{A'}$}] (child3) at (2.5,-1) {};

    \node[label={[label distance=0.8cm]-90:$|T^*_{h-1}|$}] (label1) at (-0,-1.5) {};

    \node[label={[label distance=0.8cm]-90:$|T^*_{h-1}|$}] (label1) at (-2,-1.5) {};

    \node[draw, circle, fill=black, inner sep=2pt, label={[label distance=0.8cm]-90:$N$}] (child31) at (1.4,-1.5) {};
    \node[draw, circle, fill=black, inner sep=2pt, label={[label distance=0.8cm]-90:$N$}] (child32) at (2.5,-1.5) {};
    \node[draw, circle, fill=black, inner sep=2pt, label={[label distance=0.45cm]-90:$X$}] (child33) at (3.6,-1.5) {};
    
    \draw (root2) -- (child1);
    \draw (root2) -- (child2);
    \draw (root2) -- (child3);

    \draw (child3) -- (child31);
    \draw (child3) -- (child32);
    \draw (child3) -- (child33);

    \draw[dashed] (0.85,-3.2) -- (4.3,-3.2) -- (4.3,-1.5) -- (3.05,-0.5) -- (1.95,-0.5) -- (0.85,-1.5) -- cycle;
    \node[label=above:{$U$}] (label) at (3.4,-0.8) {};
    \draw[dotted] (-1,-3.4) -- (4.6,-3.4) -- (4.6,0.3) -- (-1,0.3) -- cycle;
        
    \draw (-2,-1) -- (-2.7,-3) -- (-1.3,-3) -- cycle; 

    \draw (0,-1) -- (-0.7,-3) -- (0.7,-3) -- cycle;

    \draw (1.4,-1.5) -- (1,-3) -- (1.8,-3) -- cycle; 
    \draw (2.5,-1.5) -- (2.1,-3) -- (2.9,-3) -- cycle; 
    \draw (3.6,-1.5) -- (4,-2.5) decorate [decoration={zigzag,segment length=2mm}] {-- (3.15,-3)} -- cycle; 

\end{tikzpicture}
\caption{}\label{fig:IH1_3 ternary}
\end{subfigure}
\caption{
This figure illustrates the three possibilities of an admissible graph $A$ with $|A'|\geq 2$ of type as in Figure \ref{fig:admissible-c} (and as in Figure \ref{fig:admissible-b}, if we consider only the dotted regions). Let $U$ be the induced subgraph on the first $2N+X+1$ vertices in the eating order, indicated by the dashed region and distinguished according to the number of children of $r_{A'}$. By Observation \ref{obs:2}, $U$ is admissible in all cases. Therefore, we are able to use the induction hypothesis to embed any forest $F_1$ with $|F_1|\leq 2N+X$ into~$U$.
}
\label{fig:IH1 ternary}
\end{figure}

Otherwise, we consider $A$ as in Figure~\ref{fig:admissible-b} with $|A'|>1$. Set $N := |T^*_{h-2}|$ and $X:=|A''|$, where~$A''$ is the subgraph rooted at the lexicographically largest child of $r_{A'}$. 
We first look at the case~$|F| \leq 2N+X+1$. Consider the induced subgraph~$U$ on~$r_{A'}$, the vertices in~$A''$, and on~$u_1=r_{A''}-1$,~$u_2=r_{A''}-2$ with their descendants~$D_{u_1}, D_{u_2}$. The graph~$U$ is depicted in Figure \ref{fig:IH1 ternary} restricted to the dotted region and divided according to the number of children of~$r_{A'}$.
By Observation~\ref{obs:2},~$U$ is admissible and note that~$V\subseteq V(U)$. If the root~$r_U=r_{A'}$ is not in~$V$, we embed~$F$ into~$U$ using the induction hypothesis.
Otherwise, if~$r_U$ is in~$V$, we place an arbitrary vertex~$v$ of~$F$ at~$r_U$ and use the induction hypothesis to embed~$F\setminus v$ into~$U$.
Note that~$r_U$ is connected to every vertex in~$U$ by the~\ref{en: type 1 edges} and~\ref{en: type 2 edges} edges. In any case, we exactly eat~$V$, which proves the claim.

In the setting of Figure~\ref{fig:admissible-b}, the remaining case is when~$|A'|>1$ and~$|F|> 2N+X+1$.
Since~$A$ is as in Figure~\ref{fig:admissible-b}, we have~$|F| \leq 5N+X+2$.
By applying Lemma \ref{lem:sep-main-cons1}, we obtain a vertex~$s \in F$ and a partition~$F_1,F_2,F_3$ of~$F\setminus s$ such that,
\begin{equation}\label{partition2}
|F_1| \leq 2N+X,
\quad
|F_2| \leq 2N + \wavytriangle(F_1),
\quad
|F_3| \leq 2N + \wavytriangle(F_1 \cup F_2)+1
\quad
\text{and}
\quad
|F_1| +|F_2|\geq 2N+X.
\end{equation} 
The property that $|F_1|\leq 2N+X$ allows us to embed $F_1$ as in the previous case, using the same admissible subgraph $U$ shown in Figure~\ref{fig:IH1 ternary}. Let $A^{(1)}$ be the resulting graph after the embedding of $F_1$. If $r_{A'}$ does not have a child in $A^{(1)}$, we directly place $s$ at $r_{A'}$, so that $F_1 \cup \{s\}$ eats the first $|F_1|+1$ vertices of $A$.
Thus, by Observation~\ref{obs:1}, the remaining graph is admissible and we embed $F_2\cup F_3$ 
using the induction hypothesis. This proves the claim in that case, since the \ref{en: type 2 edges} edges connect $r_{A'}$ to every vertex in $T_{h-1}^2$. 
Otherwise, $r_{A'}$ does have a child in $A^{(1)}$ and let $u$ be the lexicographically largest one. In that case,~$A^{(1)}$ has shape as in Figure~\ref{fig:IH1 ternary} restricted to the dotted region with $X=\wavytriangle(F_1)$. 
We consider the induced subgraph $U^{(1)}$ in $A^{(1)}$ on $r_{A'}$ and on $u$, $u-1$, $u-2$ with their descendants $D_u$, $D_{u-1}$, $D_{u-2}$. By Observation~\ref{obs:2}, the graph $U^{(1)}$ is admissible and note that $|U^{(1)}|= 2N+\wavytriangle(F_1)+1$. We recall $|F_2|\leq 2N+\wavytriangle(F_1)$ given by \eqref{partition2}. Thus, we embed $F_2$ into $U^{(1)}$ using the induction hypothesis and denote the resulting graph by $A^{(2)}$.
Since $|F_1|+|F_2|\geq 2N+X$, the vertex $r_{A'}$ has no children in~$A^{(2)}$. Hence, we place $s$ at $r_{A'}$ such that $F_1\cup F_2\cup\{s\}$ is embedded at the first $|F_1|+|F_2|+1$ vertices in the eating order of $A$.
If $F_3$ is empty this proves the claim. Otherwise, we embed $F_3$ into $A^{(2)}\setminus r_{A'}$ using the induction hypothesis. The claim follows, as $r_{A'}$ is connected to every vertex in $T_{h-1}^2$ by the \ref{en: type 2 edges} edges; this completes the proof in the case that $A$ is as in Figure~\ref{fig:admissible-b}.

In what follows we assume that $A$ is as in Figure \ref{fig:admissible-c}, i.e., that it consists of the root $r_A$, an admissible subgraph $A'$ (first in the eating order) and two copies $T_{h-1}^2$ and $T_{h-1}^3$ of $T^*_{h-1}$ (second and third in the eating order).
If $|A'|=1$, we place an arbitrary vertex $v\in F$ at $r_{A'}$ and embed $F\setminus v$ into $A\setminus r_{A'}$ using the induction hypothesis. This proves the claim, since $r_{A'}$ is connected to every vertex in $T_{h-1}^2$ and $T_{h-1}^3$, due to the \ref{en: type 2 edges} edges.
Otherwise, as before, set $N := |T^*_{h-2}|$ and $X:=|A''|$, where $A''$ is the subgraph rooted at the lexicographically largest vertex among the children of $r_{A'}$.
If $|F|\leq 5N+X+2$ then we proceed exactly as in the case of Figure~\ref{fig:admissible-b};
the only difference is that $r_A$ has three instead of two children, which is depicted in Figure \ref{fig:IH1 ternary}, but does not alter any of the previously discussed steps.
So, we focus on the case where $A$ is as in Figure~\ref{fig:admissible-c} and $|F| >5N+X+2$. 
The three possible shapes of $A$ are depicted in Figure~\ref{fig:IH1 ternary}.
The size of $A$ ensures that $|F|\leq 8N+X+3$. Thus, by applying Lemma \ref{lem:sep-main-cons2}, we obtain a vertex $s_1\in F$ and a partition $F_1,F_2,F_4,\overline F$
of $F\setminus s_1$, a vertex $s_2\in\overline F$ and a partition $F_3,F_5,F_6$ of $\overline F\setminus s_2$ such that
\begin{align}\label{partition3}
|F_i| \leq 2N + \wavytriangle\left(\bigcup_{1\leq j <i} F_j \right) + I_i \text{ for }1 \le i \le 6
\quad\text{and}\quad
|F_1|+|F_2|\ge\frac32N+X,
\end{align}
where $I_i\in\{0,1\}$ is $1$ only if $|F|=8N+X+3$ and $i=6$.
Let $A^{(0)}=A$.
In the following, for $1\leq i\leq 6$, we iteratively embed $F_i$ into $A^{(i-1)}$ and denote the resulting graph after the embedding by $A^{(i)}$. We proceed in three main steps, that ultimately construct an embedding of $F$:
\begin{enumerate}
    \item\label{step1}
    \begin{enumerate}
        \renewcommand{\labelenumii}{\alph{enumii})}
        \item\label{step1a} Embed $F_1,\dots,F_{i_1}$, choosing $i_1$ minimal so that $r_{A'}$ has no child in $A^{(i_1)}$.
        \item\label{step1b} Place $s_1$ at $r_{A'}$.
    \end{enumerate}
    \item\label{step2}
    \begin{enumerate}
        \renewcommand{\labelenumii}{\alph{enumii})}
        \item\label{step2a} Embed $F_{i_1+1},\dots, F_{i_2}$, choosing $i_2$ minimal so that $r_2=r_{A'}-1$ has no child in~$A^{(i_2)}$.
        \item\label{step2b} Place $s_2$ at $r_2$.
    \end{enumerate}
    \item\label{step3} Embed the remaining forests $F_{i_2+1},\dots, F_6$. 
\end{enumerate}
We begin with Step~\ref{step1}, starting with $i=1$. 
Let $u$ be the lexicographically largest child of $r_{A'}$ in $A^{(i-1)}$.
Moreover, let $U^{(i)}$ be the induced subgraph of $A^{(i-1)}$ on~$r_{A'}$ and on~$u$,~$u-1$,~$u-2$ with descendants $D_{u}$, $D_{u-1}$, $D_{u-2}$ (see also Figure \ref{fig:IH1 ternary} for a visualization of $U^{(i)}$ and $A^{(i-1)}$ where $X=\wavytriangle(F_1\cup\dots\cup F_{i-1})$). By Observation \ref{obs:2}, $U^{(i)}$ is admissible.
Note that $|U^{(i)}|=2N+\wavytriangle(F_1\cup\dots\cup F_{i-1})+1$ ensures that $|F_i|<|U^{(i)}|$ for $1\leq i \leq 5$.
Hence, we use the induction hypothesis to embed $F_i$ into $A^{(i-1)}$ using the admissible graph $U^{(i)}$. We repeat this procedure and thereby embed $F_i$ into $A^{(i)}$ until $r_{A'}$ has no child in the resulting graph $A^{(i_1)}$ anymore. Since $|F|>5N+X+2$, this process terminates with $i_1\leq 5$. Moreover, $r_2=r_{A'}-1$ has at least one child in $A^{(i_1)}$, since $|F_{i_1}|\leq 2N+\wavytriangle(F_1\cup\dots\cup F_{i_1-1})$, by \eqref{partition3}.
Note that $r_{A'}$ is connected to every vertex in $A'$ and $T_{h-1}^2$ thanks to the \ref{en: type 1 edges} and \ref{en: type 2 edges} edges. Thus, placing $s_1$ at $r_{A'}$ yields a proper embedding of the first $i_1$ forests.

We continue with Step~\ref{step2} and $i\geq i_1+1$. Let $u'$ be the lexicographically largest child of $r_2$ in $A^{(i-1)}$.
We consider the induced subgraph $U^{(i)}$ of $A^{(i-1)}$ on $r_2$ and on $u'$, $u'-1$, $u'-2$ with their descendants $D_{u'}$, $D_{u'-1}$, $D_{u'-2}$ (see also Figure \ref{fig:IH1 ternary} restricted to the dotted region for a visualization of $U^{(i)}$ and $A^{(i-1)}$ where $X=\wavytriangle(F_1\cup\dots\cup F_{i-1})$). By Observation \ref{obs:2}, $U^{(i)}$ is admissible and $|U^{(i)}|=2N+\wavytriangle(F_1\cup\dots\cup F_{i-1})+1$ ensures that $|F_i|<|U^{(i)}|$ except for $I_6=1$, where $I_6=1$ can only happen when $|F|=8N+X+3$, which implies $i_2\leq5$.
We use the induction hypothesis to embed $F_i$ into $A^{(i-1)}$ using the admissible graph $U^{(i)}$ until $r_{2}$ has no child in $A^{(i_2)}$. 
Note that this works well with placing $s_1$ at $r_{A'}$ in Step~\ref{step1}, since the \ref{en: type 2 edges} edges connect $r_{A'}$ to every vertex in $T_{h-1}^2$ and $T_{h-1}^3$.
By \eqref{partition3}, $|F_1|+|F_2|\geq \frac{3}{2}N+X$, so that placing $s_2$ at $r_2$ provides a proper embedding of the first $i_2$ forests and $s_1$, $s_2$ thanks to the \ref{en: type 1 edges}, \ref{en: type 2 edges}, and \ref{en: type 3 edges} edges, as depicted Figure~\ref{fig:admissible backwards ternary}. 

In order to complete the proof we observe that after the first two steps, the forsets $F_1,\dots,F_{i_2}$ and the separating vertices $s_1, s_2$ are embedded at the first $|F_1|+\dots+|F_{i_2}|+2$ vertices in the eating order of~$A$. Thus, the remaining graph is admissible, by Observation~\ref{obs:1}. This allows us to carry over Step~\ref{step3}, by embedding~$F_{i_2+1}\cup\dots\cup F_6$ into the remaining admissible graph using the induction hypothesis. The proof is completed.

\paragraph{Remark}
While embedding $F$ into $A$ we used at several places \ref{en: type 1 edges} and \ref{en: type 2 edges} edges, while \ref{en: type 3 edges} edges were used only in the very last case. The reason is the lower bound $|F_1|+|F_2|\geq \frac{3}{2}N+X$ in~\eqref{partition3} that originates from Lemma \ref{lem:sep-main-cons2}. As a consequence, after embedding $F_1$ and $F_2$, up to $\lfloor\frac{1}{2}N\rfloor$ unused vertices may remain in the subgraph $A'$.
The \ref{en: type 3 edges} edges are added precisely to connect the separating vertex, which is embedded at $r_2$, to these remaining vertices. If one could ensure a stronger bound in Lemma~\ref{lem:sep-main-cons2}, for example as in Lemma~\ref{lem:sep-main-cons1}, fewer (or no) \ref{en: type 3 edges} edges would be necessary, which would improve the upper bound in Theorem \ref{thm:main}. However, there exist trees for which this lower bound cannot be improved. For example, let $N>0$ and $X=N$, and let $H$ consist of three rooted trees $H_1, H_2, H_3$, each of size $\lceil\frac{5}{2}N\rceil$ (assuming that~$\lceil\frac{5}{2}N\rceil>\frac{5}{2}N$), together with a single additional vertex $v$ only connected to the roots of $H_1, H_2, H_3$. In this sense, our construction is tight.

\begin{figure}[tb]
\centering
\begin{tikzpicture}[scale=0.9, every node/.style={transform shape}]
    \node[draw, circle, fill=black, inner sep=2pt, label={[label distance=-1mm]-280:{$r_A$}}] (root2) at (0,-0.2) {};

    \node[draw, circle, fill=black, inner sep=2pt, label={[label distance=1mm]180:$r_3$}] (child1) at (-2.6,-1) {};
    \node[draw, circle, fill=black, inner sep=2pt, label={[label distance=-2mm]-280:$r_2$}] (child2) at (0,-1) {};
    \node[draw, circle, fill=black, inner sep=2pt, label=right:{$r_{A'}$}] (child3) at (3.3,-1) {};


    \node[draw, circle, fill=black, inner sep=2pt, label={[label distance=0.8cm]-90:$N$}] (child21) at (-1.1,-1.5) {};
    \node[draw, circle, fill=black, inner sep=2pt, label={[label distance=0.8cm]-90:$N$}] (child22) at (0,-1.5) {};
    \node[draw, circle, fill=black, inner sep=2pt, label={[label distance=0.8cm]-90:$N$}] (child23) at (1.1,-1.5) {};

    \node[draw, circle, fill=black, inner sep=2pt, label={[label distance=0.45cm]-90:$Y$}] (child31) at (2.2,-1.5) {};
    \node[draw, circle, fill=black, inner sep=2pt, color=gray, label={[label distance=0.8cm,color=gray]-90:$N$}] (child32) at (3.3,-1.5) {};
    \node[draw, circle, fill=black, inner sep=2pt, color=gray, label={[label distance=0.55cm,color=gray]-90:$X$}] (child33) at (4.4,-1.5) {};

    \draw (root2) -- (child1);
    \draw (root2) -- (child2);
    \draw (root2) -- (child3);
    \draw (child2) -- (child21);
    \draw (child2) -- (child22);
    \draw (child2) -- (child23);

    \draw (child3) -- (child31);
    \draw[draw=gray] (child3) -- (child32);
    \draw[draw=gray] (child3) -- (child33);

    \draw[dashed] (child2) -- (child31);
        
    \node[label={[label distance=1.3cm]-90:$|T^*_{h-1}|$}] (label1) at (child1) {};
    \draw (-2.6,-1) -- (-3.3,-3) -- (-1.9,-3) -- cycle;

    \draw (-1.1,-1.5) -- (-1.5,-3) -- (-0.7,-3) -- cycle; 
    \draw (0,-1.5) -- (-0.4,-3) -- (0.4,-3) -- cycle; 
    \draw (1.1,-1.5) -- (0.7,-3) -- (1.5,-3) -- cycle; 

    \draw[draw=gray] (4.4,-1.5) -- (4.8,-2.6) decorate [decoration={zigzag,segment length=2mm}] {-- (4,-3)} -- cycle; 
    \draw[draw=gray] (3.3,-1.5) -- (3.7,-3) -- (2.9,-3) -- cycle; 
    \draw[draw=gray] (2.2,-1.5) -- (2.6,-3) -- (1.8,-3) -- cycle; 
    
    \draw (2.2,-1.5) -- (2.44,-2.4) decorate [decoration={zigzag,segment length=2mm}] {-- (1.85,-2.8)} -- cycle; 
    
\end{tikzpicture}

\caption{This figure illustrates the graph $A$ in Figure \ref{fig:IH1 ternary} after the embedding of forests $F_1$, $F_2$ with $|A|-(|F_1|+|F_2|)\geq 6N+Y+3$ and $Y=\protect\wavytriangle(F_1\cup F_2)$. Since $|F_1|+|F_2| \geq \frac{3}{2}N+X$, we have $Y\leq \lfloor\frac{1}{2}N\rfloor$. Because of the \ref{en: type 3 edges} edges, $r_2$ is connected to every vertex in the subgraph of size $Y$ indicated by the dashed edge. Crucially, this allows us to embed the remaining forests $F_3,\dots,F_6$ in the visible graph and place the second separating vertex $s_2$ at $r_{2}$ after $s_1$ is placed at $r_{A'}$.}
\label{fig:admissible backwards ternary}
\end{figure}

\section{Other proofs}\label{sec:other proofs}
This section contains the missing proofs of Lemmas \ref{lem:sep-main}, \ref{lem:size}, \ref{lem:sep-main-cons1}, \ref{lem:sep-main-cons2}.

\subsection{Proof of Lemma \ref{lem:sep-main}}\label{sec: proof of lem:sep-main}

Let $F$ be a forest with $|F|\geq M+1$, where $M \ge 2m$ and $m \ge 0$. 
We may assume that $F$ is a tree by adding edges, if this is necessary. 
Using Lemma \ref{lem:sep-chung}, we choose a vertex $s \in F$ and a forest $F_3$ such that $|F_3|$ is maximal with respect to $m \leq |F_3| \leq M$ and $F_3 \subseteq F \setminus s$.
Then there are two cases to distinguish.
    \begin{itemize}
        \item If $|F_3|=M$, then the forests $F_1 = F\setminus(F_3 \cup s)$, $F_2 = \emptyset$ and $F_3$ obviously satisfy the conclusion.
        \item If $|F_3| < M$, then $F\setminus (F_3\cup s)$ consists of at least two disjoint trees, since otherwise we could increment $|F_3|$ by choosing the unique neighbor of $s$ which is not in $F_3$ as the separating vertex;
        this would contradict the maximality of $|F_3|$.
        We choose $F_2$ to be the smallest tree  in $F\setminus (F_3 \cup s)$ and define $F_1= F\setminus (F_3 \cup s \cup F_2)$, so that, by construction, $|F_2| \leq |F_1|$.
        Further, $|F_2| + |F_3| > M$, since otherwise $F_2 \cup F_3$ would contradict the maximality of $|F_3|$. So,
        \[
        |F_1| = |F|-1-(|F_2| + |F_3|) \leq |F|-1-M-1.
        \]
    \end{itemize}

\subsection{Proof of Lemma \ref{lem:size}}\label{sec:size}
We first count the vertices in $T_{h,d}^*$. Since there are $d^\ell$ vertices on each level $0\leq \ell\leq h$,
\[
|T^*_{h,d}|=\sum_{0 \leq \ell\leq h} d^\ell=\frac{d^{h+1}-1}{d-1}.
\]
This allows us to also count the edges in $T_{h,d}^*$. Recall the three types of edges of a vertex $v\in V(T^*_{h,d})$.
For $v$ at level $0\leq \ell\leq h$, we get
\begin{itemize}
    \item $|D_v| = |T^*_{h-\ell,d}|-1$ edges of \ref{en: type 1 edges},
    \item $(d-1)|T^*_{h-\ell,d}|$ edges of \ref{en: type 2 edges},
    \item and $\lfloor |T^*_{h-(\ell+1),3}|/2\rfloor$ edges of \ref{en: type 3 edges}, in the case $d=3$.
\end{itemize}
We say that these edges are \emph{added} by vertex $v$ and denote the number of them by $e_{h,d}(\ell)$. 
Therefore, 
\[
e_{h,d}(\ell)=d|T^*_{h-\ell,d}|-1+\mathbf{1}_{d = 3}\left\lfloor|T^*_{h-(\ell+1),3}|/2\right\rfloor\le d|T^*_{h-\ell,d}|+\frac{\mathbf{1}_{d = 3}}{6}|T^*_{h-\ell,3}| = \left(d+\frac{\mathbf{1}_{d = 3}}{6}\right)|T^*_{h-\ell,d}|.
\]
The number of edges $e(h,d)$ in $T_{h,d}^*$ is given by the sum over all vertices and their number of added edges. 
By combining the previous statements, we obtain
\[
e(h,d)=\sum_{0\leq \ell\leq h} e_{h,d}(\ell)d^\ell\leq\left(d+\frac{\mathbf{1}_{d = 3}}{6}\right)\sum_{0\leq \ell\leq h}\frac{d^{h-\ell+1}-1}{d-1}d^\ell\leq\frac{6d+\mathbf{1}_{d = 3}}{6\ln d}\ln((d-1)|T^*_{h,d}|+1)|T^*_{h,d}|.
\]
Finally, we derive a handy presentation of $n = |U_{n,d}|$, which then allows us to achieve the desired upper bound on the number of edges. Let $h\geq 0$ be such that $|T_{h-1,d}^*|<n\leq |T_{h,d}^*|$. 
Recall that $U_{n,d}$ is the induced subgraph of $T_{h,d}^*$ given by the $n$ vertices that are eaten last. 
Observe that for any vertex at level $i>0$ with position $n\geq p\geq1$ in the eating order, there exists $\alpha_{i}\in \{0,\dots,d-1\}$ such that the vertex with position $p+\alpha_{i}|T^*_{h-i,d}|+1$ in the eating order is at level $i-1$. Applying this observation iteratively to $U_{n,d}$ starting with the first vertex in the eating order at level $\ell^*>0$ and terminating with the root at level $0$, we get
\[
n=1+ \sum_{0\leq i \leq \ell^*-1}(\alpha_{\ell^*-i}|T^*_{h-(\ell^*-i),d}|+1) = 1 + \sum_{1\leq \ell\leq\ell^*}(\alpha_\ell|T^*_{h-\ell,d}|+1),
\]
where we substituted $\ell=\ell^*-i$.
Using this expression, the edges of $U_{n,d}$ are given by the edges of all the complete subgraphs, the edges of the root and the edges of the $\ell^*$ vertices $r_\ell$, which are roots of subgraphs at level $1\leq\ell\leq \ell^*$. We know that the number of edges of $r_\ell$ is bounded by $e_{h,d}(\ell)$ and thus the number of edges of $U_{n,d}$ is bounded by
\begin{align*}
\sum_{1\leq\ell\leq \ell^*}(\alpha_\ell e(h-\ell,d)+e_{h,d}(\ell)) +O(n)
&\leq \frac{6d+\mathbf{1}_{d = 3}}{6\ln d}\ln((d-1)n)\sum_{1\leq\ell\leq\ell^*}\alpha_\ell|T^*_{h-\ell,d}|+O(n) \\
&\leq \frac{6d+\mathbf{1}_{d = 3}}{6\ln d}n\ln n+O(n).
\end{align*}
This proves the claim, since $\frac{6d+\mathbf{1}_{d = 3}}{6\ln d}=\frac{19}{6\ln 3}$ in the case $d=3$ and $\frac{d}{\ln d}$ otherwise.


\subsection{Proof of Lemma \ref{lem:sep-main-cons1}}\label{subsec:proof of sep-main-cons1}

Let $X>0$ and $N\geq X$. Let $F$ be a forest with $2N+X+2\leq |F|\leq 5N+X+2$. We define
\[
    m:=\max\big\{0,|F|-(4N+X+1)\big\}
    \quad 
    \text{and}
    \quad 
    M:=|F|-(2N+X+1).
\]
We aim to apply Lemma~\ref{lem:sep-main}. Clearly $|F|\geq M+1$, so it remains to verify that $2m\leq M$. If $|F|\leq 4N+X+1$, then $2m\leq M$ holds trivially since $M\geq1$. Moreover, if  $4N+X+1<|F|\leq 5N+X+2$, then
\[
    2m = 2|F|-2(4N+X+1) = M+|F|-(6N+X+1)\leq M.
\]
Therefore, we apply Lemma \ref{lem:sep-main} to get a vertex $s \in F$ and a partition $F_1$, $F_2$, $F_3$ of $F\setminus s$ such that
\[
    m \leq |F_3| \leq M,
    \quad
    |F_1| \leq |F| -1 -M,
    \quad
    \text{and}
    \quad
    |F_2| \leq |F_1|.
\]
By definition of $M$, this implies
\[
|F_1| \leq|F|-1-M= 2N+X
\quad \text{and}\quad
|F_1| +|F_2|=|F|-(|F_3|+1) \geq |F|-(M+1)= 2N+X.
\]
We next argue that $|F_2| \leq 2N +\wavytriangle(F_1)$. In the case $|F_1| \leq N+X$,
the inequality directly follows since $|F_2| \leq |F_1|\leq 2N$.
Otherwise $N+X<|F_1|\leq 2N+X$ so that $|F_1|\geq2N+X-\wavytriangle(F_1)$. Hence,
\[
    |F_2| = |F|-(|F_1|+|F_3|+1)\leq|F|-(2N+X-\wavytriangle(F_1)+m+1)\le 2N + \wavytriangle(F_1).
\]
It remains to show $|F_3|\leq 2N+\wavytriangle(F_1\cup F_2)+1$.
Since $|F_1| +|F_2| \geq 2N+X$, we have $|F_1|+|F_2|\geq 3N+X-\wavytriangle(F_1\cup F_2)$ and
the proof finishes by observing that
\[
    |F_3| = |F|-(|F_1|+|F_2|+1) \leq |F|-(3N+X-\wavytriangle(F_1 \cup F_2)+1) \leq 2N+\wavytriangle(F_1\cup F_2)+1.
\]

\subsection{Proof of Lemma \ref{lem:sep-main-cons2}}\label{subsec:proof of sep-main-cons2}

Let $X>0$ and $N\geq X$. Let $F$ be a forest with $5N+X+3\leq |F|\leq 8N+X+3$.
We define
\[
    m_1:=|F|-(5N+X+2)
    \quad 
    \text{and}
    \quad
    M_1:=|F|-(2N+X+1).
    \]
We aim to apply Lemma~\ref{lem:sep-main}. Since $|F|\leq 8N+X+3$, we obtain
\[
    2m_1 = 2|F|-2(5N+X+2) = M_1+|F|-(8N+X+3) \leq M_1.
\]
In particular, $2m_1\leq M_1$, and clearly $|F|\geq M_1+1$.
Thus, by Lemma \ref{lem:sep-main}, there exists a vertex $s_1 \in F$ and a partition $F_1,H_2,H_3$ of $F\setminus s_1$ such that
\begin{align}
\label{sep1}
    m_1 \leq |H_3| \leq M_1,
    \quad
    |F_1| \leq |F| -1 -M_1,
    \quad
    \text{and}
    \quad
    |H_2| \leq |F_1|.
\end{align}
Using the definition of $M_1$, this implies
\begin{align}\label{size F_1}
    |F_1|\leq |F|-1-M_1 = 2N+X.
\end{align}
In what follows we will distinguish between two cases according to $|F_1| + |H_2|$.
In particular, in each case we define further appropriate subgraphs $F_2, \dots, F_6$ of $F$ that fulfill the required size constraints.  
We will write $\wavytriangle_1:=\wavytriangle(F_1)$ and whenever $F_2,\dots,F_5$ are defined,
\[
\wavytriangle_i := \wavytriangle(F_1\cup\dots\cup F_i), \quad 1\leq i\leq 5.
\]

\paragraph{Case $|F_1| + |H_2| \leq 4N+X$.}
Set
\[
    F_2 = H_2,
    \quad
    F_4=\emptyset
    \quad
    \text{and}
    \quad
    \overline{F}=H_3.
\]
We first verify that $|F_2|\leq 2N+\wavytriangle_1$. Indeed, if 
$|F_2|>2N+\wavytriangle_1$, then
$2N+\wavytriangle_1<|F_1|\leq 2N+X$, by~\eqref{sep1} and~\eqref{size F_1}. By definition of the rest, this implies $|F_1|=2N+X-\wavytriangle_1$, which contradicts $|F_1|+|F_2|\leq 4N+X$. Therefore,
\[
    |F_2|\leq 2N+\wavytriangle_1.
\]
Recall that $|H_3|\leq M_1$, by~\eqref{sep1}. Then, using that $M_1=|F|-(2N+X+1)$ we obtain
\begin{align*}
    |F_1|+|F_2| = |F|-(|H_3|+1)\geq |F|-(M_1+1)  =2N+X
\end{align*}
and so, $|F_1|+|F_2|\geq \frac{3}{2}N+X$, as required.
For later reference, this also implies
\begin{align}\label{size F_1+F_2}
|F_1|+|F_2| \geq  3N+X-\wavytriangle_2.
\end{align}
In what follows, we further distinguish cases according to the size of $\overline{F}$.
\begin{itemize}
    \item If $|\overline{F}|\leq 2N+\wavytriangle_2+1$, let $s_2$ be any vertex in $\overline{F}$. 
    Set $F_3 = \overline{F} \setminus s_2$ such that
    \[
    |F_3|= |\overline{F}|-1\leq 2N+\wavytriangle_2.
    \]
    Choosing $F_5=F_6 = \emptyset$ then proves the claim.
    \item In the remaining case $|\overline{F}|\geq 2N+\wavytriangle_2+2$, using~\eqref{size F_1+F_2} and~$|F|\leq 8N+X+3$, we obtain
    \[
    |\overline{F}|=|H_3|=|F|-(|F_1|+|F_2|+1)\leq 5N+\wavytriangle_2+2.
    \]
    In particular, $ 2N+\wavytriangle_2+2\leq |\overline{F}|\leq 5N+\wavytriangle_2+2$.
    We apply Lemma \ref{lem:sep-main-cons1} to $\overline{F}$ with $\wavytriangle_2$ for $X$ (and $N$ for $N$). This yields a vertex $s_2\in \overline{F}$ and a partition $F_3, F_5, F_6$ of $\overline{F}\setminus s_2$ such that
    \begin{align*}
        |F_3| \leq 2N+\wavytriangle_2,
        \quad
        |F_5| \leq 2N + \wavytriangle_4,
        \quad
        |F_6| \leq 2N + \wavytriangle_5+1,
        \quad
        \text{and}
        \quad
        |F_3| +|F_5|\geq 2N+\wavytriangle_2.
    \end{align*}
    This directly shows that $F_3$ and $F_5$ satisfy the required conditions.
    If $|F_6|\leq 2N+\wavytriangle_5$, this proves the claim. Otherwise $|F_6|= 2N+\wavytriangle_5+1$. Since $|F_3|+|F_5|\geq 2N+\wavytriangle_2$, we have $|F_3|+|F_5|\geq 3N+\wavytriangle_2-\wavytriangle_5$. Together with \eqref{size F_1+F_2} this shows
    \[
        |F|=|F_1|+|F_2|+|F_3|+|F_5|+|F_6|+2 \geq 8N+X+3.
    \]
    Hence, $|F|=8N+X+3$, so $I_6=1$ and $|F_6|=2N+\wavytriangle_5+I_6$, which proves the claim.
\end{itemize}
    
\paragraph{Case $|F_1| +|H_2| > 4N+X$.}
Set 
\[
    F_2=\emptyset
    \quad 
    \text{and}
    \quad
    \overline H=H_2\cup H_3.
\]
Note that $|F_1| +|H_2| > 4N+X$ is only possible if $|F_1| \geq \frac{3}{2}N+X$, since $|H_2| \leq |F_1|$ by~\eqref{sep1}.
Thus,~$F_1$ and $F_2$ fulfill all size constraints on them.
We first treat two cases where Lemma \ref{lem:sep-main-cons1} directly applies and finally show the statement for large $|\overline{H}|$.
\begin{itemize}
    \item Assume that $|F_1|=2N+X$.
    Let $F_4=\emptyset$ and $\overline{F}=\overline{H}$. 
    The bounds on $|F|$ and $|\overline{F}|=|F|-(|F_1|+1)$ imply
    \[
    3N+2\leq |\overline{F}|\leq 6N+2.
    \]
    We apply Lemma \ref{lem:sep-main-cons1} to $\overline{F}$ with $N$ for $X$ (and $N$ for $N$) to obtain $s_2\in\overline{F}$ and a partition $F_3,F_5,F_6$ of $\overline{F}\setminus s_2$ such that
    \[
    |F_3| \leq 3N,
    \quad
    |F_5| \leq 2N + \wavytriangle_3,
    \quad
    |F_6| \leq 2N + \wavytriangle_5+1,
    \quad
    \text{and}
    \quad
    |F_3| +|F_5|\geq 3N.
    \]
    Hence, $F_3$ and $F_5$ satisfy the size conditions.
    Note that $|F_3|+|F_5|\geq 3N$ implies $|F_3|+|F_5|\geq 4N-\wavytriangle_5$.
    Therefore, $|F_6| = 2N+\wavytriangle_5+1$ is only possible in the case $|F|=8N+X+3$. Thus, $|F_6|\leq 2N+\wavytriangle_5+I_6$, completing the proof.
    \item Next, assume that $|\overline{H}|<5N+\wavytriangle_1+2$ and $|F_1|<2N+X$.
    Let $F_4=\emptyset$ and $\overline{F}=\overline H$. 
    Since $|\overline{F}|=|F|-(|F_1|+1)$ and $|F_1|\leq 2N+X-\wavytriangle_1$,
    \[
    3N+\wavytriangle_1+2 \leq |\overline{F}|\leq 5N+\wavytriangle_2+1.
    \]
    Thus, Lemma \ref{lem:sep-main-cons1} applied to the forest $\overline{F}$ with  $\wavytriangle_1$ for $X$ and $N$ for $N$ yields $s_2\in \overline{F}$ and a partition $F_3,F_5,F_6$ of $\overline{F}\setminus s_2$ such that
    \[
    |F_3| \leq 2N+\wavytriangle_1,
    \quad
    |F_5| \leq 2N + \wavytriangle_3,
    \quad
    |F_6| \leq 2N + \wavytriangle_5+1,
    \quad
    \text{and}
    \quad
    |F_3| +|F_5|\geq 2N+\wavytriangle_1.
    \]
    Thus, $F_3,F_5$ directly fulfill their size constraints. Note that $|F_3|+|F_5|\geq 2N+\wavytriangle_1$ implies $|F_3|+|F_5|\geq 3N+\wavytriangle_1-\wavytriangle_5$. If $|F_6|=5N+\wavytriangle_5+1$, this contradicts $|\overline{F}|\leq 5N+\wavytriangle_1+1$. Hence, $|F_6|\leq 2N+\wavytriangle_5$, which proves the claim.
    \item Finally we treat the case $|\overline H|\ge 5N +\wavytriangle_1+2$ and $|F_1|<2N+X$. 
    Let $F_4=H_2$. We will split off a little forest from $H_3$ which is larger than $\wavytriangle_1$.
    Our assumptions $4N+X<|F_1|+|F_4|$, $|F_1|\geq |F_4|$ (by~\eqref{sep1}) and $|F_1|<2N+X$ imply 
    \begin{align}
    \label{size F_1 + F_4}
        4N+X<|F_1|+|F_4| <5N+X
        \quad
        \text{so that}
        \quad
        |F_1|+|F_4|=5N+X-\wavytriangle(F_1\cup F_4).
    \end{align}
    This implies $\wavytriangle(F_1\cup F_4)\geq 2 \wavytriangle_1$, because
    \[
    \wavytriangle(F_1\cup F_4)=5N+X-(|F_1|+|F_4|) \ge 5N+X-2|F_1|= 5N+X-2(2N+X-\wavytriangle_1)\geq 2 \wavytriangle_1.
    \]
    Thus, we can apply Lemma \ref{lem:sep-main} to $\overline F=H_3$ with $m_2:=\wavytriangle_1$, $M_2:=\wavytriangle(F_1\cup F_4)$ to obtain $s_2 \in \overline F$ and a partition $F_6$, $F_5$, $F_3$ of $\overline F\setminus s_2$ such that
    \begin{align}\label{sep2}
    m_2 \leq |F_3| \leq M_2,
    \quad
    |F_6| \leq |\overline F| -1 -M_2,
    \quad
    \text{and}
    \quad
    |F_5| \leq |F_6|.
    \end{align}
    Consequently, $|F_3|\leq M_2\leq N$. 
    Moreover, using $|F_3|\leq M_2$ together with~\eqref{size F_1 + F_4}, we obtain
    \begin{align}\label{size F_1+F_3+F_4}
        |F_1|+|F_4|+|F_3|\leq 5N+X-\wavytriangle(F_1\cup F_4)+M_2 = 5N+X.
    \end{align}
    Note that $|F_1|=2N+X-\wavytriangle_1$, since $\frac{3}{2}N+X\leq |F_1|<2N+X$. Hence, 
    \[
    |F_1|+|F_3|\geq 2N+X-\wavytriangle_1+m_2\geq 2N+X 
    \quad \text{and}\quad
    |F_1|+|F_3|\geq 3N+X-\wavytriangle_3.
    \]
    Since $|F_4|\leq5N+X-(|F_1|+|F_3|)$, by~\eqref{size F_1+F_3+F_4}, this further implies
    \[
    |F_4|\leq 5N+X-(|F_1|+|F_3|) \leq 5N+X-(3N+X-\wavytriangle_3) = 2N+\wavytriangle_3.
    \]
    Moreover, $|F_5|\leq 2N$. Indeed, using $|F_5|\leq |F_6|$ from~\eqref{sep2} and $|F_1|+|F_4|>4N+X$ from~\eqref{size F_1 + F_4}, we obtain
    \[
    2|F_5|\leq |F_5|+|F_6| \leq |F|-(|F_1|+|F_4|+2) \leq 4N.
    \]
    Finally, we verify that $|F_6|\leq 2N+\wavytriangle_5+I_6$.
    By~\eqref{sep2}, 
    \[
    |F_3|+|F_5|= |\overline{F}|-(|F_6|+1) \geq |\overline{F}|-(|\overline{F}|-M_2)=\wavytriangle(F_1\cup F_4).
    \]
    Therefore, using~\eqref{size F_1 + F_4},
    \[
    |F_6|= |F|-(|F_1|+|F_4|+|F_3|+|F_5|+2) \leq |F|-(6N+X-\wavytriangle_5+2) \leq 2N+\wavytriangle_5+I_6.
    \]
\end{itemize}

\section{Proof of Theorem \ref{thm:main-treewidth}}\label{sec: results treewidth}

This section is organized as follows. We first strengthen the lower bound of Kaul and Wood~\cite{kaulwood2025} by refining their argument. Then we establish the upper bound by performing an appropriate blow-up of our construction for trees. Since many of the arguments are similar, we keep the exposition shorter and highlight the differences. We also assume $w = o(n)$ throughout, as this is the only case in which the statement of the theorem is not trivial.

\subsection{Lower bound}

Let $U$ contain an isomorphic copy of every graph with treewidth $w$. We assume that $n$ is large enough so that $\lfloor n/(2w+1)\rfloor \geq 1$. For $j\in \{1,\dots, \lfloor n/(2w+1)\rfloor\}$, let $S_j$ be the complete bipartite graph $K_{w,\lfloor n/j\rfloor-w}$ where the vertices are partitioned in two classes $A_j$ and $B_j$, with $|A_j|=w$ and $|B_j| =\lfloor n/j\rfloor-w$. Moreover, let $H_j$ be the disjoint union of $j$ copies $S_j^1,\dots,S_j^j$ of $S_j$. Since $|H_j| \leq n$ and $\text{tw}(H_j) \leq w$, the graph $U$ contains a subgraph $\mathcal{H}_j$ isomorphic to $H_j$.
Let $\mathcal{A}_j^* = \mathcal{A}_j^1 \cup \dots  \cup \mathcal{A}_j^j$ and $\mathcal{B}_j^* = \mathcal{B}_j^1 \cup \dots  \cup \mathcal{B}_j^j$ where $\mathcal{A}_j^i$ and $\mathcal{B}_j^i$ are the set of vertices in $\mathcal{H}_j$ corresponding to $A_j^i$ and $B_j^i$, respectively.

We proceed by coloring edges and vertices in $U$. Starting with $j=1$, we color every vertex in $\mathcal{A}_1^*$ and every edge in $\mathcal{H}_1$. In step $j\geq2$, we consider $\mathcal{H}_j$ where some vertices and edges might already have been colored due to the preceding $j-1$ steps. We order the uncolored vertices in $\mathcal{A}_j^*$ by the highest number of uncolored incident edges in $\mathcal{H}_j$ and color the first $w$ vertices in this order. Finally, we color all edges in $\mathcal{H}_j$ incident to newly colored vertices.
\vspace{2mm}

\noindent\textbf{Claim:} For each step $j$ the number of newly colored edges is at least $w (\lfloor n/j \rfloor -2w)$.
\vspace{1mm}

\noindent This proves the lower bound since the claim asserts that the number of colored edges is at least
\[
\sum_{1\leq j\leq \lfloor n/(2w+1)\rfloor} w (\lfloor n/j \rfloor -2w) \geq w\int_{1}^{\lfloor n/(2w+1)\rfloor} (n/j-2w-1) \, dj = nw \ln(n/w) -O(nw).
\]
It remains to prove the claim. For $j=1$ there are obviously $w(n-w)$ edges in $\mathcal{H}_1$. Let $j\geq2$. After step $j-1$, there are $(j-1)w$ colored vertices and every colored edge contains at least one colored vertex. For $i \in \{1,\dots,j\}$, we say that $\mathcal{B}_j^i$ is \textit{blocked} if the set contains at least $w+1$ colored vertices. Let $b_j$ be number of blocked vertex sets. Then there are at most $(j-1)w-b_j(w+1)$ colored vertices, and hence at least $jw-(j-1)w+b_j(w+1) = w(b_j+1)+b_j$ uncolored vertices, in $\mathcal{A}_j^*$. If $\mathcal{B}_j^i$ is \textit{not} blocked, we say that every uncolored vertex in $\mathcal{A}_j^i$ is \textit{available}. Notice that every available vertex is contained in at least $\lfloor n/j \rfloor-2w$ disjoint uncolored edges in $\mathcal{H}_j$. The number of available vertices is at least
\[
w(b_j+1)+b_j - b_jw = w + b_j \geq w. 
\]
Hence, by coloring $w$ available vertices and all adjacent edges in $\mathcal{H}_j$, at least $w(\lfloor n/j \rfloor-2w)$ edges are newly colored, as claimed.

\subsection{Upper bound}

\begin{figure}[tb]
    \centering
\begin{tikzpicture}[font=\sffamily,scale=0.85, every node/.style={transform shape}]

"\clip (0cm,0.7cm) rectangle (14.8cm,6.5cm);
\coordinate (v0) at (7.4,5.5);
\coordinate (v1) at (2.42307692307692,3.86666666666667);
\coordinate (v2) at (7.4,3.86666666666667);
\coordinate (v3) at (12.3769230769231,3.86666666666667);
\coordinate (v11) at (0.83076923076923,2.23333333333333);
\coordinate (v12) at (2.42307692307692,2.23333333333333);
\coordinate (v13) at (4.01538461538462,2.23333333333333);
\coordinate (v21) at (5.80769230769231,2.23333333333333);
\coordinate (v22) at (7.4,2.23333333333333);
\coordinate (v23) at (8.99230769230769,2.23333333333333);
\coordinate (v31) at (10.7846153846154,2.233333333333333);
\coordinate (v32) at (12.3769230769231,2.233333333333333);
\coordinate (v33) at (13.9692307692308,2.23333333333333);
\draw (v0) -- (v1);
\draw (v0) -- (v2);
\draw (v0) -- (v3);
\draw (v1) -- (v11);
\draw (v1) -- (v12);
\draw (v1) -- (v13);
\draw (v2) -- (v21);
\draw (v2) -- (v22);
\draw (v2) -- (v23);
\draw (v3) -- (v31);
\draw (v3) -- (v32);
\draw (v3) -- (v33);
\node[draw, circle, minimum size=0.9cm, inner sep=0pt, fill=white, label=right:{\footnotesize$(\emptyset,i)$}] (root2) at (v0) {\footnotesize$K_{w+1}$};
\node[draw, circle, minimum size=0.9cm, inner sep=0pt, fill=white, label=right:{\footnotesize$(1,i)$}] (root2) at (v1) {\footnotesize$K_{w+1}$};
\node[draw, circle, minimum size=0.9cm, inner sep=0pt, fill=white, label=right:{\footnotesize$(2,i)$}] (root2) at (v2) {\footnotesize$K_{w+1}$};
\node[draw, circle, minimum size=0.9cm, inner sep=0pt, fill=white, label=right:{\footnotesize$(3,i)$}] (root2) at (v3) {\footnotesize$K_{w+1}$};
\node[draw, circle, minimum size=0.9cm, inner sep=0pt, fill=white, label=below:{\footnotesize$(11,i)$}] (root2) at (v11) {\footnotesize$K_{w+1}$};
\node[draw, circle, minimum size=0.9cm, inner sep=0pt, fill=white, label=below:{\footnotesize$(12,i)$}] (root2) at (v12) {\footnotesize$K_{w+1}$};
\node[draw, circle, minimum size=0.9cm, inner sep=0pt, fill=white, label=below:{\footnotesize$(13,i)$}] (root2) at (v13) {\footnotesize$K_{w+1}$};
\node[draw, circle, minimum size=0.9cm, inner sep=0pt, fill=white, label=below:{\footnotesize$(21,i)$}] (root2) at (v21) {\footnotesize$K_{w+1}$};
\node[draw, circle, minimum size=0.9cm, inner sep=0pt, fill=white, label=below:{\footnotesize$(22,i)$}] (root2) at (v22) {\footnotesize$K_{w+1}$};
\node[draw, circle, minimum size=0.9cm, inner sep=0pt, fill=white, label=below:{\footnotesize$(23,i)$}] (root2) at (v23) {\footnotesize$K_{w+1}$};
\node[draw, circle, minimum size=0.9cm, inner sep=0pt, fill=white, label=below:{\footnotesize$(31,i)$}] (root2) at (v31) {\footnotesize$K_{w+1}$};
\node[draw, circle, minimum size=0.9cm, inner sep=0pt, fill=white, label=below:{\footnotesize$(32,i)$}] (root2) at (v32) {\footnotesize$K_{w+1}$};
\node[draw, circle, minimum size=0.9cm, inner sep=0pt, fill=white, label=below:{\footnotesize$(33,i)$}] (root2) at (v33) {\footnotesize$K_{w+1}$};"

\draw[loosely dotted, rounded corners=9pt,line width=0.35mm] ($(v0)+(1.4,0.6)$) -- ($(v0)+(-0.6,0.6)$) -- ($(v1)+(-0.7,0.5)$) -- ($(v11)+(-0.7,0.5)$) -- ($(v11)+(-0.7,-1.2)$) -- ($(v13)+(0.7,-1.2)$) -- ($(v13)+(0.7,0.5)$) --  ($(v1)+(1.2,-0.2)$) -- ($(v0)+(1.4,-0.2)$) -- cycle;

\end{tikzpicture}
\caption{\small The figure depicts the vertex set of $U^w_{n_1}$ with $n_1=13(w+1)$, which arises from $U_{13,3}$ by replacing every vertex with a clique  $K_{w+1}$. Given a vertex $v\in U_{13,3}$, the labels after the blow-up are $(v,i)$, for $1\leq i\leq w+1$, denoted next to the circle. 
Moreover, removing vertices following the eating order, also determines the vertex set for other graphs $U^w_n$ of height $3$. For instance, after eating the first $8(w+1)$ vertices, the dotted region shows the vertex set of $U^w_{n_1}$ with $n_1= 5(w+1)$.
}
\label{fig:universal treewidth}
\end{figure}

The proof of the upper bound is given in three steps.
First, we construct the graphs $U_{n}^w$ using the framework of Section~\ref{sec:proof-strategy}. Next, we adapt the Lemma~\ref{lem:sep-main} on separating vertices in trees, along with its consequence, Lemma~\ref{lem:sep-main-cons1} and~\ref{lem:sep-main-cons2}. Finally, we prove that the graphs $U_{n}^w$ are universal for graphs on $n$ vertices with treewidth $w$.

Let $n^* = \lceil n/(w+1)\rceil$.
We start with the universal graph $U_{n^*,3}$, constructed in Section~\ref{sec:proof-strategy}, and perform a blow-up as follows. We replace every vertex $v$ in $U_{n^*,3}$ by a complete graph  $K_{w+1}$ of size $w+1$ with vertex labels $(v,1),\dots,(v,w+1)$. Two vertices $(v,i)\neq (u,j)$ are connected if $v$ and $u$ are connected in $U_{n^*,3}$ or $v=u$. Moreover, we naturally extend the eating order $``\succ"$ by
\[
(v,i) \succ (u,j) \Longleftrightarrow v\succ u \text{ or } v=u \text{ and }i <j.
\]
The constructed graph has $wn^* \geq n$ vertices and we remove vertices following the eating order to obtain the graph $U_{n}^w$ on $n$ vertices, see Figure~\ref{fig:universal treewidth} as an example. For a vertex $(v,i) \in U_n^w$, let $D_{(v,i)}$ denote the set of all \textit{descendants} and $C_{(v,i)}$ the \textit{clique} of $(v,i)$ given by
\[
D_{(v,i)} = \left\{(u,j) \in V(U_{n}^w) : u \in D_v \right\}
\quad
\text{and}
\quad 
C_{(v,i)} =\left\{(u,j) \in V(U_{n}^w) : u=v, i\neq j  \right\}.
\]
Also the operations $\pm$ are naturally adapted from Section~\ref{sec:proof-strategy} by $(v,i)\pm a := (v\pm a,i)$.
Using this notation, the edges of $U_n^w$ connect every vertex $(v,i)\in V(U_n^w)$ to
\begin{typesofedgestw}
    \item\label{en: type 0 edges-treewidth} every vertex in $C_{(v,i)}$;
    \item\label{en: type 1 edges-treewidth} every vertex in $D_{(v,i)}$;
    \item\label{en: type 2 edges-treewidth} $(v,i)-1, (v,i)-2$ and every vertex in $C_{(v,i)-1}\cup C_{(v,i)-2}$ and $D_{(v,i)-1}\cup D_{(v,i)-2}$;
    \item\label{en: type 3 edges-treewidth} every vertex in the half (rounded down and with respect to the maximum number in such vertex set) of $\{(u,w+1)\}\cup C_{(u,w+1)} \cup D_{(u,w+1)}$ that is eaten last, where $u$ is the lexicographically smallest child of $v+1$.
\end{typesofedgestw}
The following lemma shows that $U_n^w$ has the desired number of edges, which we directly deduce from the number of edges in $U_{n^*,3}$.
\begin{lemma}\label{lem: size treewidth}
    The graph $U_n^w$ has $\frac{19}{6\ln3}(w+1)n\ln(n/w)+O(wn)$ edges.
\end{lemma}
\begin{proof}
    In the construction, we start with a universal graph $U_{n^*,3}$ on $n^* =\lceil n/(w+1) \rceil$ vertices. Lemma~\ref{lem:size} shows that $U_{n^*,3}$ has $\frac{19}{6\ln3}n^*\ln n^* + O(n^*)$ edges. Due to the blow-up, every edge in $U_{n^*,3}$ occurs $(w+1)^2$ times in the graph $U_n^w$. Moreover, every clique (except for one, which might be partially eaten up) adds $\binom{w+1}{2}$ edges. Since we add $n^*$ cliques, the number of edges is bounded by
    \[
    (w+1)^2\left(\frac{19}{6\ln3}n^*\ln n^* + O\left(n^*\right)\right) + \binom{w+1}{2} n^* \leq \frac{19}{6\ln3} (w+1)n \ln(n/w) + O(wn).
    \]
\end{proof}
It remains to show that $U_n^w$ is universal to complete the proof of Theorem \ref{thm:main-treewidth}. 
Let $h\geq 0$. As a further preparation we also define the graph $T^w_{h}$ as the blow-up of the graph $T_{h,3}^*$ constructed in Section~\ref{sec:proof-strategy}. This means that $T^w_{h}$ is the perfect ternary tree with height $h$, where every vertex is replaced by a clique of size $w+1$ and the edges of \ref{en: type 0 edges-treewidth} -- \ref{en: type 3 edges-treewidth} are added as described above. 
As in Section \ref{sec:ternary trees}, we establish universality for a broader class of so called admissible graphs.
\begin{definition}
    A graph $A\neq \emptyset$ is called \emph{admissible}, if there exists $h$ such that $A$ isomorphic to the induced subgraph on the last $|A|$ vertices in the eating order of $T^w_{h}$. The eating order on $A$ is thereby naturally inherited from $T^w_{h}$.
\end{definition}
As in Section~\ref{sec:proof-strategy}, there is a recursive description of admissible graphs.
\begin{remark}\label{rem:admissible-ternary}
    Let $A$ be an admissible graph. Then there exists $T^w_{h}$ for some $h$ such that one of the following holds.
    \begin{enumerate}
        \item There exists $1\leq c \leq w+1$ such that $A$ is given by the last $c$ vertices in the eating order of $T^w_{h}$.
        \item There exists an admissible subgraph $A'$ of $T^w_{h-1}$ such that $A$ is isomorphic to one of the three possible ways shown in Figure~\ref{fig:admissible-treewidth}. Further, $A$ inherits the eating order from $T^w_{h-1}$ as follows. The vertices of $A'$ are eaten first, given by the order on $A'$. Next, the vertices of the up to two copies of $T^w_{h-1}$ are eaten one after another, given by the eating order on $T^w_{h-1}$. Finally, the vertices at the root position are eaten by the given eating order on them.
    \end{enumerate}
\end{remark}
\begin{figure}[ht!]
    \centering
    \begin{subfigure}[t]{0.3\textwidth}
        \centering
        \begin{tikzpicture}[scale=0.85, every node/.style={transform shape}]
            \draw[very thick] (-2.15,-4.5) -- (-1.5,-2.5) -- (-0.85,-4.1) decorate [decoration={zigzag,segment length=2mm}] {-- cycle};
            \node[draw, circle, minimum size=0.9cm, inner sep=0pt, fill=white, label=right:{$R_A$}] (root2) at (0,-1) {\footnotesize$K_{w+1}$};
            \node[draw, circle, minimum size=0.9cm, inner sep=0pt, fill=white, label={[label distance=0.65cm]-90:$A'$}] (child1) at (-1.5,-2.5) {};

            \fill[white] (-2.15,-4.5) -- (-1.5,-2.5) -- (-0.85,-4.1) decorate [decoration={zigzag,segment length=2mm}] {-- cycle};
            \node[circle, minimum size=0.9cm, inner sep=0pt, label={[label distance=0.6cm]-90:$A'$}] (child1) at (-1.5,-2.5) {};
            \draw (root2) -- (child1);
        \end{tikzpicture}
        \caption{}
        \label{fig:admissible-a-treewidth}
    \end{subfigure}
    %
    \begin{subfigure}[t]{0.3\textwidth}
        \centering
        \begin{tikzpicture}[scale=0.85, every node/.style={transform shape}]
            \draw[very thick] (-2.4,-4.5) -- (-1.75,-2.5) -- (-1.1,-4.5) -- cycle;
            \node[draw, circle, minimum size=0.9cm, inner sep=0pt, fill=white] (child1) at (-1.7,-2.5) {};
            \fill[white] (-2.4,-4.5) -- (-1.75,-2.5) -- (-1.1,-4.5) -- cycle;
            \node[circle, minimum size=0.9cm, inner sep=0pt, label={[label distance=0.9cm]-90:$T^w_{h-1}$}] (child1) at (-1.65,-2.5) {};

            \draw[very thick] (-0.65,-4.5) -- (0,-2.5) -- (0.65,-4.1) decorate [decoration={zigzag,segment length=2mm}] {-- cycle};
            \node[draw, circle, minimum size=0.9cm, inner sep=0pt, fill=white, label=right:{$R_A$}] (root2) at (0,-1) {\footnotesize$K_{w+1}$};
            
            \node[draw, circle, minimum size=0.9cm, inner sep=0pt, fill=white, label={[label distance=0.65cm]-90:$A'$}] (child2) at (0,-2.5) {};

            \fill[white] (-0.65,-4.5) -- (0,-2.5) -- (0.65,-4.1) decorate [decoration={zigzag,segment length=2mm}] {-- cycle};

            \node[ circle, minimum size=0.9cm, inner sep=0pt, label={[label distance=0.6cm]-90:$A'$}] (child2) at (0,-2.5) {};
            \draw (root2) -- (child1);
            \draw (root2) -- (child2);
        \end{tikzpicture}
        \caption{}
        \label{fig:admissible-b-treewidth}
    \end{subfigure}
    %
    \begin{subfigure}[t]{0.3\textwidth}
        \centering
        \begin{tikzpicture}[scale=0.85, every node/.style={transform shape}]
            \draw[very thick] (-2.4,-4.5) -- (-1.75,-2.5) -- (-1.1,-4.5) -- cycle;
            \draw[very thick] (-0.65,-4.5) -- (0,-2.5) -- (0.65,-4.5) -- cycle;  
            \draw[very thick] (1.1,-4.5) -- (1.75,-2.5) -- (2.4,-4.1) decorate [decoration={zigzag,segment length=2mm}] {-- cycle}; 
            \node[draw, circle, minimum size=0.9cm, inner sep=0pt, fill=white, label=right:{$R_A$}] (root2) at (0,-1) {\footnotesize$K_{w+1}$};
            \node[draw, circle, minimum size=0.9cm, inner sep=0pt, fill=white, label={[label distance=0.9cm]-90:$T^w_{h-1}$}] (child1) at (-1.75,-2.5) {};
            
            \node[draw,circle, minimum size=0.9cm, inner sep=0pt,fill=white, label={[label distance=0.9cm]-90:$T^w_{h-1}$}] (child2) at (0,-2.5) {};
            \node[draw, circle, minimum size=0.9cm, inner sep=0pt, fill=white, label={[label distance=0.6cm]-90:$A'$}] (child3) at (1.75,-2.5) {};

            \fill[white] (-2.4,-4.5) -- (-1.75,-2.5) -- (-1.1,-4.5) -- cycle;
            \fill[white] (-0.65,-4.5) -- (0,-2.5) -- (0.65,-4.5) -- cycle;  
            \fill[white] (1.1,-4.5) -- (1.75,-2.5) -- (2.4,-4.1) decorate [decoration={zigzag,segment length=2mm}] {-- cycle}; 

            \node[circle, minimum size=0.9cm, inner sep=0pt] (child1) at (-1.75,-2.5) {};
            \node[circle, minimum size=0.9cm, inner sep=0pt,  label={[label distance=0.9cm]-90:$T^w_{h-1}$}] at (-1.7,-2.5) {};
            
            \node[circle, minimum size=0.9cm, inner sep=0pt] (child2) at (0,-2.5) {};
            \node[circle, minimum size=0.9cm, inner sep=0pt, label={[label distance=0.9cm]-90:$T^w_{h-1}$}] at (0.1,-2.5) {};
            
            \node[circle, minimum size=0.9cm, inner sep=0pt,  label={[label distance=0.6cm]-90:$A'$}] (child3) at (1.75,-2.5) {};
            \draw (root2) -- (child1);
            \draw (root2) -- (child2);
            \draw (root2) -- (child3);
            
        \end{tikzpicture}
        \caption{}
        \label{fig:admissible-c-treewidth}
    \end{subfigure}
    \caption{\small The admissible graph $A$ is given by its $w+1$ root vertices $R_A$, the admissible subgraph $A'$ of $T^w_{h-1}$ and up to two copies of $T^w_{h-1}$. The vertex sets and the graph structure are indicated in the figure. Moreover, although not depicted, as an induced subgraph, $A$ contains all edges of $T^w_{h-1}$ on~$V(A)$.
    }
    \label{fig:admissible-treewidth}
\end{figure}
We also extend the definition of the \emph{rest} of a forest to graphs with treewidth~$w$.
\begin{definition}
    Let $0<X\leq N$. For a graph $G$ with treewidth at most $w$ let the \emph{rest} $\wavytriangle(G)=\wavytriangle_{N,X}(G)$ be defined by 
    \[
    \wavytriangle(G)=X-|G| \text{ for }|G|<X
    \quad\text{and}\quad
    \wavytriangle(G)=N-x \text{ for }|G|\geq X,
    \]
    where, if $|G|\geq X$, then $0\leq x<N$ is unique such that $|G|=x+kN+X$ for some $k\geq 0$.
\end{definition}
As in Section~\ref{sec:proof-strategy}, an inductive argument allows us to establish universality for all admissible graphs.
We therefore focus on finding admissible \emph{sub}graphs. The definition ensures that removing vertices following the eating order preserves the property of being admissible, as captured in the following observation.
\begin{observation}\label{obs:1-tw}
    Let $U$ be the resulting graph after removing the first $0\leq t <|A|$ vertices in the eating order of an admissible graph $A$. Then $U$ is again admissible.
\end{observation}
Moreover, the (recursive) structure of admissible graphs described in Remark~\ref{rem:admissible-ternary} yields several additional useful admissible subgraphs.
\begin{observation}\label{obs:2-tw} 
    Let $A$ be an admissible graph and $(r,w+1), (c,w+1)\in A$ be such that $c$ is a child of~$r$.
    Moreover, let $t\in\{1,2\}$ such that $c-t$ is lexicographically larger than $c$ and set
    \[
    D=\bigcup_{0\leq i\leq t}\Big(\{(c,w+1)-i\}\cup C_{(c,w+1)-i}\cup D_{(c,w+1)-i}\Big).
    \]
    Then the induced subgraph of $A$ on the vertices $\{(r,w+1)\}\cup C_{(r,w+1)} \cup D$ is admissible. 
\end{observation}
Building on the two observations, we show that any admissible graph $A$ is universal for all graphs $G$ with treewidth at most $w$. Moreover, if $|G|\leq |A|-(w+1)$, we are able to embed $G$ at the first $|G|$ vertices in the eating order of $A$, leading to the following definition.
\begin{definition}
    For any graph $G$ with treewidth at most $w$ and admissible graph $A$, a mapping $\lambda:V(G)\to V(A)$ is an \emph{embedding} of $G$ into $A$ if 
    \begin{itemize}
        \item $\lambda(V(G))$ consists of the first $|G|$ vertices in the eating order of $A$, and 
        \item for every edge $uv$ in $G$, the vertices $\lambda(u)$ and $\lambda(v)$ are adjacent in $A$.
    \end{itemize}
\end{definition}
With this definition at hand, we are now able to state our main result regarding universal graphs for graphs with treewidth $w$. 
\begin{lemma}\label{lem: main-treewidth} 
    Let $A$ be an admissible graph and $G$ be a graph with $\text{tw}(G)\leq w$ and $|G|\leq |A|-(w+1)$. Then there exists an embedding $\lambda$ of $G$ into $A$. In particular, $A\setminus\lambda(V(G))$ is admissible.
\end{lemma}
As the following remark shows, this immediately implies that every graph $U^w_n$ is universal. Combined with Lemma~\ref{lem: size treewidth}, this yields the upper bound in Theorem~\ref{thm:main-treewidth}.
\begin{remark}
    Let $A$ be any admissible graph with $n$ vertices, in particular $A$ could be some $U^w_n$.
    Let $G$ be any graph with $n$ vertices and $\text{tw}(G) =w$. Note that $\text{tw}(G)=w$ implies $w+1\leq n$. We choose an arbitrary set of vertices $S\subset G$ with $|S| =w+1$ and embed $G\setminus S$ into $A$ using Lemma \ref{lem: main-treewidth}. 
    Note that the vertices in $R_A$ are connected to each other by the \ref{en: type 0 edges-treewidth} edges and to \emph{every} other vertex in $A$ via the \ref{en: type 1 edges-treewidth} edges. Thus, we conclude by placing $S$ at $R_A$.
\end{remark}

\subsubsection{Splitting graphs with treewidth $w$}

Fundamentally, every graph with treewidth $w$ has a \textit{normal} tree-decomposition, meaning that every bag has size $w+1$ and the intersection of any two adjacent bags has size $w$, see for example~\cite{Woodnormaltreedecomp}. Kaul and Wood \cite{kaulwood2025} use this fact to generalize the separator lemma by Chung and Graham. Let $G$ be a graph with tree-decomposition $(B_x : x\in T)$. For any vertex $z\in T$ and forest $F \subset T$, let $G(F,z)$ be the induced subgraph of $G$ on the vertex set $(\bigcup_{x\in F}B_x)\setminus B_z$. 

\begin{lemma}\label{lem:aux-sep-treewidth}
    Let $t\in \mathbb{N}_0$ and $G$ be a graph with $\text{tw}(G)= w$ and $|G| \geq t +w+1$. For every normal tree-decomposition $(B_x : x\in T)$ of $G$, there exists a vertex $z\in T$ such that for some forest $F\subset T\setminus z$,
    \[
    t \leq |G(F, z)|\leq 2t.
    \]
\end{lemma}
Building on this lemma, we extend our results about separating vertices in trees, i.e.,~Lemmas~\ref{lem:sep-main},~\ref{lem:sep-main-cons1} and~\ref{lem:sep-main-cons2}. 
For completeness, we also include concise proofs, even though the arguments are similar to those in the tree setting.
The main difference is that instead of removing single vertices, we remove sets of vertices. For graphs with treewidth~$w$ the size of these sets is $w+1$. We adjust the notation and say that for a given graph $G$, the (possibly empty) graphs $G_1,\dots,G_t$ form a partition of $G$ if they are disjoint and $G_1 \cup \dots \cup G_t =G$. Note that if $\text{tw}(G)\leq w$, then also $\text{tw}(G_i)\leq w$ for $1\leq i\leq t$.
\begin{lemma}\label{lem:sep-treewidth}
    Let $0 \leq m$, $2m \leq M$ and $G$ be a graph with $\text{tw}(G)\leq w$ and $|G|\geq M+w+1$. Then there exists a set of vertices $S \subset G$ with $|S|=w+1$ and a partition $G_1, G_2, G_3$ of $G\setminus S$ such that 
    \[
    m \leq |G_3| \leq M,
    \quad
    |G_1| \leq |G| -w-1 -M,
    \quad
    \text{and}
    \quad
    |G_2| \leq |G_1|.
    \]
\end{lemma}
\begin{proof}
    We may assume that $G$ has treewidth $w$ by adding edges, if this is necessary. Let $(B_x:x\in T)$ be a normal tree-decomposition of $G$. Using Lemma \ref{lem:aux-sep-treewidth}, we choose a vertex $z \in T$ and a forest $F_3 \subset T\setminus z$ such that $G_3=G(F_3, z)$ is maximal with respect to $m\leq |G_3| \leq M$. Let $S=B_z$, which implies $|S|=w+1$. Then there are two cases to distinguish.
    \begin{itemize}
        \item If $|G_3| = M$, then the graphs $G_1 = G\setminus (G_3\cup S)$, $G_2 = \emptyset$ and $G_3$ obviously satisfy the conclusion of the lemma.
        \item If $|G_3|<M$, then $T\setminus (F_3\cup z)$ consists of at least two disjoint trees. Otherwise, since the tree-decomposition is normal, we could increment $|G_3|$ by choosing $S=B_{z'}$, where $z'$ is the unique neighbor of $z$ which is not in $F_3$; this would contradict the maximality of $|G_3|$. Hence, there are at least two non-empty disjoint connected components in $G\setminus (G_3\cup S)$. We choose $G_2$ to be the smallest connected component in $G\setminus (G_3\cup S)$ and define $G_1 = G\setminus (G_3\cup S\cup G_2)$, so that, by construction, $|G_1|\geq |G_2|$. Further $|G_2|+|G_3| >M$, since otherwise $G_2 \cup G_3$ would contradict the maximality of $|G_3|$. So,
        \[
        |G_1| = |G|-w-1-(|G_2|+|G_3|) \leq |G|-w-M-2.
        \]
    \end{itemize}
\end{proof}
The next lemma handles the case where we use one set of separating vertices of size $w+1$. This reflects the setting where, given a graph $G$ with $\text{tw}(G)\leq w$ and $2N+X+w+2 \leq |G|\leq 5N +X+2w+2$, we remove $w+1$ vertices to obtain a partition $G_1$, $G_2$, $G_3$ of $G\setminus S$. Controlling the sizes of the $G_i$ will later allow us to apply an inductive embedding strategy.

\begin{lemma}\label{cor: sep treew1}
    Let $X>0$ and $N\geq \max\{X,w+1\}$. Let $G$ be a graph with $\text{tw}(G)\leq w$ and $2N+X+w+2 \leq |G|\leq 5N +X+2w+2$. Then there exists a set of vertices $S\subset G$ with $|S|=w+1$ and a partition $G_1,G_2, G_3$ of $G\setminus S$ such that
    \[
    |G_1| \leq 2N+X, \quad 
    |G_2| \leq 2N + \wavytriangle(G_1), \quad
    |G_3| \leq 2N+\wavytriangle(G_1\cup G_2)+w+1
    \quad 
    \text{and}\quad 
    |G_1|+|G_2| \geq 2N+X.
    \]
\end{lemma}
\begin{proof}
    We set 
    \[
    m:=\max\big\{0,|G|-w-1-4N-X\big\}
    \quad \text{and}\quad
    M:=|G|-w-1-2N-X.
    \]
    Note that $2m \leq M$ and clearly $|G|\geq M+w+1$. Hence, we apply Lemma \ref{lem:sep-treewidth} to get a set of vertices $S\subset G$ with $|S|=w+1$ and a partition $G_1,G_2,G_3$ of $G\setminus S$ such that 
    \[
    m \leq |G_3| \leq M,
    \quad
    |G_1| \leq |G| -w-1 -M,
    \quad
    \text{and}
    \quad
    |G_2| \leq |G_1|.
    \]
    The definition of $M$ directly implies
    \[
    |G_1|\leq |G| -w-1 -M=2N+X
    \quad\text{and}\quad
    |G_1|+|G_2| = |G|-(|G_3|+w+1)\geq 2N+X.
    \]
    We next argue that $|G_2|\leq 2N+\wavytriangle(G_1)$
    If $|G_1|\leq N+X$, the inequality holds since $|G_2|\leq |G_1|\leq 2N$. Otherwise $X+N<|G_1|\leq 2N+X$, such that
    \[
    |G_2| = |G|-(|G_1|+|G_3|+w+1) \leq |G|-(2N+X-\wavytriangle(G_1)+m+w+1)\leq 2N + \wavytriangle(G_1).
    \]
    It remains to show $|G_3| \leq 2N+\wavytriangle(G_1\cup G_2)+w+1$.
    Since $|G_1|+|G_2| \geq 2N+X$ we have $|G_1|+|G_2| \geq 3N+X-\wavytriangle(G_1\cup G_2)$ and using the upper bound on $|G|$ we obtain
    \[
    |G_3| = |G|-(|G_1|+|G_2|+w+1)\leq |G|-(3N + X -\wavytriangle(G_1\cup G_2)+w+1 )\leq 2N+\wavytriangle(G_1\cup G_2)+w+1.
    \]
\end{proof}
If $|G|\geq 5N+X+2w+3$, it is again possible to prepare the inductive argument by removing \emph{two} sets of separating vertices of size $w+1$. 

\begin{lemma}\label{cor: sep treew2}
    Let $X>0$ and $N\geq \max\{X,w+1\}$. Let $G$ be a graph with $\text{tw}(G)\leq w$ and $5N+X+2w+3 \leq |G| \leq 8N+ X+3w+3$. Let $I_6 = c$, if $|G| = 8N + X +2w+2+c$ for $1\leq c \leq w+1$ and $I_i=0$ else. Then there exist a set of vertices $S_1 \subset G$ with $|S_1|=w+1$ and a partition $G_1,G_2,G_4, \overline{G}$ of $G\setminus S_1$, a set of vertices $S_2 \subset \overline{G}$ with $|S_2|=w+1$ and a partition $G_3, G_5,G_6$ of $\overline{G}\setminus S_2$ such that
    \[
    |G_i| \leq 2N +\wavytriangle\left(\bigcup_{1\leq j<i} G_j\right)+I_i \quad \text{for } i\geq1,
    \quad 
    \text{and}
    \quad
    |G_1| +|G_2| \geq \frac{3}{2}N+X.
    \]
\end{lemma}

\begin{proof}
We set 
\[
m_1=|G|-(5N+X+2w+2)
\quad \text{and}\quad
M_1 = |G| -(2N+X+w+1).
\]
Note that $2m_1\leq M_1$ and clearly $|G|\geq M_1+w+1$. Thus, we apply Lemma \ref{lem:sep-treewidth} to get a set of vertices $S_1 \subset G$ with $|S_1|=w+1$ and a partition $G_1, H_2, H_3$ of $G\setminus S_1$ such that
\begin{equation}\label{sep1-tw}
m_1 \leq |H_3| \leq M_1,
\quad
|G_1| \leq |G| -w-1 -M_1,
\quad
\text{and}
\quad
|H_2| \leq |G_1|.
\end{equation}
The definition of $M_1$ directly implies $|G_1| \leq  2N+X$.
We distinguish between two cases according to $|G_1|+|H_2|$.
\paragraph{Case $|G_1|+|H_2| \leq 4N+X$.}
Set
\[
G_2=H_2,
\quad
G_4=\emptyset
\quad \text{and}\quad
\overline{G}=H_3.
\]
Note that $|G_2|\leq 2N+\wavytriangle(G_1)$, since $|G_2|\leq |G_1|$ by \eqref{sep1-tw}. Moreover, $|G_1|+|G_2| = |G|-(|H_3|+w+1)  \geq 2N+X$. 
Thus, $G_1$ and $G_2$ satisfy the required conditions. We further distinguish two cases according to the size of $\overline{G}$.
\begin{itemize}
    \item If $|\overline{G}|\leq 2N+\wavytriangle(G_1\cup G_2)+w+1$, let $S_2\subset\overline{G}$ be any vertex set with $|S_2|=w+1$. Choosing $G_3=\overline{G}\setminus S_2$ and $G_5=G_6 =\emptyset$, then proves the claim.
    \item It remains to consider $|\overline{G}| \geq 2N+\wavytriangle(G_1\cup G_2) + w+2$. Since $|G_1|+|G_2|\geq 2N+X$ and $|G|\leq 8N+X+3w+3$, we obtain
    \[
    |\overline{G}|= |H_3| = |G|-(|G_1|+|G_2|+w+1) \leq 6N+2w+2.
    \]
    In particular, $2N+\wavytriangle(G_1\cup G_2)+w+2\leq |\overline{G}|\leq 6N+2w+2$. Thus, applying Lemma \ref{cor: sep treew1} to $\overline{G}$ and $\wavytriangle(G_1\cup G_2)$ for $X$ (and $N$ for $N$) yields a set of separating vertices $S_2$ and a suitable partition $G_3, G_5, G_6$ of $\overline{G}\setminus S_2$.
\end{itemize}
\paragraph{Case $|G_1|+|H_2| >4N+X$.} 
Set
\[
G_2=\emptyset 
\quad\text{and}\quad
\overline{H} = H_2\cup G_3.
\]
We obtain $|G_1|\geq X+\frac{3}{2}N$, since $|G_1|+|H_2| >4N+X$ and $|H_2|\leq |G_1|$ by~\eqref{sep1-tw}. Thus, $G_1$ and $G_2$ satisfy the required conditions.
Before dealing with large $\overline{G}$, we handle two cases by directly applying Lemma~\ref{cor: sep treew1}.
\begin{itemize}
    \item If $|G_1|=2N+X$, set $G_4=\emptyset$ and $\overline{G}=\overline{H}$. The bounds on $|G|$ and $|\overline{G}|=|G|-(|G_1|+w+1)$ show
    \[
    3N+w+2\leq |\overline{G}|\leq 6N+2w+2.
    \]
    Thus, Lemma \ref{cor: sep treew1} applied to $\overline{G}$ and $N$ for $X$ (and $N$ for $N$) yields a set of separating vertices $S_2$ and a suitable partition $G_3,G_5,G_6$ of $\overline{G}\setminus S_2$.
    \item Next, assume that $|\overline{H}| < 5N+\wavytriangle(G_1)+2w+2$ and $|G_1|<2N+X$. 
    Set $G_4 =\emptyset$ and $\overline{G}=\overline{H}$. Since $|\overline{G}|=|G|-(|G_1|+w+1)$ and $|G_1|\leq2N+X-\wavytriangle(G_1)$, we obtain
    \[
    3N+\wavytriangle(G_1)+w+2 \leq |\overline{G}|\leq 5N+\wavytriangle(G_1)+2w+1.
    \]
    Thus, Lemma \ref{cor: sep treew1} applied to $\overline{G}$ and $\wavytriangle(G_1)$ for $X$ (and $N$ for $N$) yields a set of separating vertices $S_2$ and a suitable partition $G_3,G_5,G_6$ of $\overline{G}\setminus S_2$.
    \item Finally, we consider the case $|\overline{H}|\geq 5N+\wavytriangle(G_1)+2w+2$ and $|G_1|<2N+X$. Set $G_4 = H_2$. We will split off a little graph from $H_3$ which is larger than $\wavytriangle(G_1)$. We observe
    \begin{equation}\label{part2-tw}
        4N+X<|G_1|+|G_4|\leq 2|G_1|<5N+X
        \quad\text{so that}\quad
        |G_1|+|G_4|=5N+X-\wavytriangle(G_1\cup G_4).
    \end{equation}
    This shows $\wavytriangle(G_1\cup G_4)\geq 2 \wavytriangle(G_1)$, since
    \[
    \wavytriangle(G_1\cup G_4)=5N+X-|G_1|-|G_4| \geq 5N+X-2|G_1| = 5N+X-2(2N+X-\wavytriangle(G_1))\geq 2\wavytriangle(G_1).
    \]
    Thus, we apply Lemma \ref{lem:sep-treewidth} to $\overline{G}=H_3$ with $m_2:=\wavytriangle(G_1)$, $M_2 := \wavytriangle(G_1\cup G_4)$ to obtain $S_2 \subset \overline{G}$ and a partition $G_6, G_5,G_3$ of $\overline{G}\setminus S_2$ such that 
    \begin{equation*}\label{sep2-tw}
    m_2 \leq |G_3| \leq M_2, \quad 
    |G_6| \leq |\overline{G}|-w-1-M_2,\quad 
    \text{and}\quad 
    |G_5|\leq |G_6|.
    \end{equation*}
    This implies $|G_3|\leq \wavytriangle(G_1\cup G_4)\leq N$. 
    Since $|G_1|=2N+X-\wavytriangle(G_1)$, we obtain 
    \[
    |G_1|+|G_3|\geq 2N+X-\wavytriangle(G_1)+m_2\geq 2N+X
    \quad \text{and}\quad
    |G_1|+|G_3|\geq 3N+X-\wavytriangle(G_1\cup G_2\cup G_3).
    \]
    Together with $|G_1|+|G_4|+|G_3|\leq 5N+X-\wavytriangle(G_1\cup G_4)+M_2=5N+X$, it follows that
    \[
    |G_4|\leq 5N+X-(|G_1|+|G_3|)\leq 5N+X-(3N+X-\wavytriangle(G_1\cup G_2\cup G_3)) =2N+\wavytriangle(G_1\cup G_2\cup G_3).
    \]
    Moreover, $|G_1|+|G_3|+|G_4|\geq 5N+X-\wavytriangle(G_1\cup \dots\cup G_4)$.
    Since $|G_5|\leq |G_6|$ and $w+1\leq N$, we obtain
    \[
    2|G_5|\leq |G_5|+|G_6|\leq |G|-(|G_1|+|G_3|+|G_4|+2w+2)\leq 4N+\wavytriangle(G_1\cup \dots\cup G_4).
    \]
    Thus, $|G_5|\leq 2N+\wavytriangle(G_1\cup \dots\cup G_4)$. Finally, $|G_6|$ satisfies its constraint. Indeed $|G_3|+|G_5|=|\overline{G}|-(|G_6|+w+1)\geq M_2$ which, together with \eqref{part2-tw}, implies
    \[
    |G_6|=|G|-(|G_1|+|G_3|+|G_4|+|G_5|+w+2)\leq 2N+\wavytriangle(G_1\cup\dots\cup G_5)+I_6.
    \]
\end{itemize}
\end{proof}

\subsubsection{Proof of Lemma \ref{lem: main-treewidth}}

The proof relies on the splitting lemmas from the previous section and is similar to the tree case with adjustments to handle the blown-up structure. We proceed by induction on $|A|\in \mathbb{N}$. For $h<2$, every admissible graph is complete. Thus, let $h\geq 2$ and assume that the statement holds  for all admissible graphs $A'$ with $|A'|<|A|$ and graphs $G'$ with $\text{tw}(G')\leq w, |G'|\leq |A'|-w-1$. 
Let $G$ be a graph with $\text{tw}(G) \leq w$ and $|G|\leq |A|-w-1$, and let $V\subset A$ be the first $|G|$ vertices in the eating order of $A$.
We distinguish three cases according to the structure of $A$ given by Figure~\ref{fig:admissible-treewidth}.

Assume that $A$ has shape as in Figure \ref{fig:admissible-a-treewidth}, consisting of the root vertices $R_A$ and an admissible subgraph $A'$. For $|G|\leq |A'|-w-1$, we apply the induction hypothesis to $G$ and $A'$, which proves the claim. Otherwise, we choose an arbitrary vertex set $S\subset G$ with $|S|=|A'|-w-1$.
Note that $|G|\leq |A|-w-1$ implies $|S|\leq w+1$.
Thus, we apply the induction hypothesis to embed $G\setminus S$ into $A'$ leaving only the root vertices $R_{A'}$ of $A'$. Then, we place $S$ at the first $|S|$ vertices in $R_{A'}$. 
This process embeds $G$ exactly at $V$ as illustrated in Figure~\ref{fig:IH-treewidth}.
Since the vertices $R_{A'}$ are connected to each other by \ref{en: type 0 edges-treewidth} edges and to every other vertex in $A'$ via \ref{en: type 1 edges-treewidth} edges, this yields a proper embedding.

Next, assume that $A$ has shape as in Figure \ref{fig:admissible-b-treewidth}, consisting of the root vertices $R_A$, one subgraph $T^w_{h-1}$ and another admissible subgraph $A'$, where $A'$ is first in the eating order. If $A'$ only consists of root vertices, i.e.~$|A'| \leq w+1$, we place $|A'|$ arbitrary vertices $S\subset G$ at $A'$ and embed $G\setminus S$ into $A\setminus A'$ using the induction hypothesis. 
This yields a proper embedding, since the vertices in $A'$ are connected to each other by \ref{en: type 0 edges-treewidth} edges and all other vertices in $T^w_{h-1}$ via \ref{en: type 1 edges-treewidth} edges.

\begin{figure}[tb]
    \centering
    \begin{subfigure}[t]{0.3\textwidth}
        \centering
        \begin{tikzpicture}[scale=0.85, every node/.style={transform shape}]
            \draw[very thick] (-2.15,-4.5) -- (-1.5,-2.5) -- (-0.85,-4.2) decorate [decoration={zigzag,segment length=2mm}] {-- cycle};
            \node[draw, circle, minimum size=0.9cm, inner sep=0pt, fill=white, label=right:{$R_A$}] (root2) at (0,-1) {\footnotesize$K_{w+1}$};
            \node[draw, circle, minimum size=0.9cm, inner sep=0pt, fill=white, label=right:{$R_{A'}$}] (child1) at (-1.5,-2.5) {};
            \fill[white] (-2.15,-4.5) -- (-1.5,-2.5) -- (-0.85,-4.2) decorate [decoration={zigzag,segment length=2mm}] {-- cycle};
             \node[circle, minimum size=0.9cm, inner sep=0pt,label={[label distance=0.6cm]-90:$A'$}] at (-1.5,-2.5) {};
            
            \draw (root2) -- (child1);
            
            \draw[dashed] (-2.5,-4.8) -- (-0.2,-4.8) -- (-0.2,-1.9) -- (-2.5,-1.9) -- cycle;
        \end{tikzpicture}
        \caption{}
        \label{fig:IH-a-treewidth}
    \end{subfigure}
    %
    \begin{subfigure}[t]{0.3\textwidth}
        \centering
        \begin{tikzpicture}[scale=0.85, every node/.style={transform shape}]
            \node[] at (0,-4.8) {};
            \node[draw, circle, minimum size=0.9cm, inner sep=0pt, fill=white, label=right:{$R_A$}] (root2) at (0,-1) {\footnotesize$K_{w+1}$};
            \node[draw, circle, minimum size=0.9cm, inner sep=0pt, fill=white, label=right:{$R_{A'}$}] (child1) at (-1.5,-2.5) {\footnotesize$K_{w+1}$};
            \draw (root2) -- (child1);
            \draw[dashed] (-2.1,-3.1) -- (-0.2,-3.1) -- (-0.2,-1.9) -- (-2.1,-1.9) -- cycle;
        \end{tikzpicture}
        \caption{}
        \label{fig:IH-b-treewidth}
    \end{subfigure}
    %
    \begin{subfigure}[t]{0.3\textwidth}
        \centering
        \begin{tikzpicture}[scale=0.85, every node/.style={transform shape}]
            \node[] at (0,-4.7) {};

            \node[draw, circle, minimum size=0.9cm, inner sep=0pt, fill=white, label=right:{$R_A$}] (root2) at (0,-1) {\footnotesize$K_{w+1}$};
            \node[circle, minimum size=0.9cm, inner sep=0pt, fill=white,
      draw=none, label=right:{$R_{A'}$}] (child1) at (-1.5,-2.5) {};

\def\r{0.45} 
\def\a{225}   
\def\b{320}  

\path (child1.center) ++(\a:\r) coordinate (A);
\path (child1.center) ++(\b:\r) coordinate (B);

\draw (B) arc[start angle=\b,end angle=\a+360,radius=\r]; 

\draw[decorate,decoration={zigzag,segment length=4pt,amplitude=2pt}] (A) -- (B);
            \draw (root2) -- (child1);
        \end{tikzpicture}
        \caption{}
        \label{fig:IH-c-treewidth}
    \end{subfigure}
    \caption{\small Illustration of the embedding of a graph $G$ with $|A'|-w-1<|G|\leq |A|-w-1$ into an admissible graph $A$ as in Figure \ref{fig:admissible-a-treewidth}. Consider a vertex set $S\subset G$ such that $|S|=|A'|-w-1$. We embed $G\setminus S$ into $A'$ using the induction hypothesis, depicted in (a), leaving only the root vertices of $A'$ shown in (b). We conclude by placing $S$ at the first $|S|$ vertices in the eating order of $R_{A'}$. This leaves either some root vertices of $A'$, as shown in (c), or only $R_A$. In both cases $G$ is exactly embedded at $V$ and the resulting graph is admissible. }
    \label{fig:IH-treewidth}
\end{figure}

Otherwise, we consider the case where $A$ is given by Figure \ref{fig:admissible-b-treewidth} and $|A'|>w+1$. Then $A'$ consists of an admissible subgraph $A''$ (first in the eating order) and up to two copies of $T^w_{h-2}$.
Set $N:= |T^w_{h-2}|$ and $X:= |A''|$.
Clearly, $N\geq \max\{X,w+1\}$. Let $(r_{A''},w+1)\in R_{A''}$.
We first consider the case $|G|\leq 2N+X+w+1$.
Let $U$ be the induced subgraph on $R_{A'}$, the vertices in $A''$, and on $u_1 =(r_{A''},w+1)-1$, $u_2 =(r_{A''},w+1)-2$ with cliques $C_{u_1}$, $C_{u_2}$ and descendants $D_{u_1}$, $D_{u_2}$.
The graph $U$ is admissible, by Observation~\ref{obs:2-tw}, and $V\subset U$.
If $V\cap R_{A'} = \emptyset$, we embed $G$ into $U$ using the induction hypothesis. Otherwise, if $|V\cap R_{A'}| =c\geq 1$, we place $c$ arbitrary vertices $S\subset G$ at the first $c$ vertices eaten in $R_{A'}$ and embed $G\setminus S$ into $U$ by the induction hypothesis. Thus, $G$ is embedded at $V$ and the vertices in $R_{A'}$ are connected to every other vertex in $U$ thanks to the \ref{en: type 0 edges-treewidth} -- \ref{en: type 2 edges-treewidth} edges.

In the setting of Figure~\ref{fig:admissible-b-treewidth}, we are left with the case $|G| > 2N+X+w+1$ and $|A'|>w+1$. The size of $A$ implies that $|G| \leq 5N+X+2w+2$. Applying Lemma \ref{cor: sep treew1} yields a set of vertices $S\subset G$ with $|S|=w+1$ and a partition $G_1,G_2,G_3$ of $G\setminus S$ such that, 
\[
    |G_1| \leq 2N+X, \quad 
    |G_2| \leq 2N + \wavytriangle(G_1), \quad
    |G_3| \leq 2N+\wavytriangle(G_1\cup G_2)+w+1
    \quad 
    \text{and}\quad 
    |G_1|+|G_2| \geq 2N+X.
\]
We embed $G_1$ using the admissible graph $U$ exactly as above. If $G_1$ eats $A'\setminus R_{A'}$ completely, we place $S$ at $R_{A'}$ so that the embedding of $G_1\cup S$ eats the first $|G_1|+w+1$ vertices in the eating order. Hence, the remaining graph is admissible, by Observation~\ref{obs:1-tw}, and we embed $G_2\cup G_3$ into it using the induction hypothesis. 
Otherwise, let $A^{(1)}$ be the resulting graph after the embedding of $G_1$. We embed $G_2$ using the admissible graph of size $2N+\wavytriangle(G_1)$ given by $R_{A'}$ and the three subgraphs of $A^{(1)}$ rooted at level two, that are eaten first. 
Since $|G_1|+|G_2| \geq 2N+X$, the embedding of $G_1\cup G_2$ eats $A'\setminus R_{A'}$ completely. Therefore, we place $S$ at $R_{A'}$, which leaves an admissible graph, by Observation~\ref{obs:1-tw}. Thus, $G_3$ can be  embedded into the remaining graph using the induction hypothesis. All this yields a proper embedding, since $G$ eats $V$ and the vertices in $R_{A'}$ are connected to every other vertex in $A'$ and $T^w_{h-1}$ due to the \ref{en: type 0 edges-treewidth} -- \ref{en: type 2 edges-treewidth} edges.

Lastly, assume that $A$ has shape as in Figure \ref{fig:admissible-c-treewidth} with root vertices $R_{A}$, an admissible subgraph $A'$ (first in the eating order), and two copies $T^2_{h-1}$ and $T^3_{h-1}$ of $T^w_{h-1}$ (second and third in the eating order). 
If $|A'|\leq w+1$, we place arbitrary $|A'|$ vertices $S\subset G$ at $A'$ and finish the proof by embedding $G\setminus S$ into $A\setminus A'$ via the induction hypothesis. 
This proves the claim, since the vertices in $A'$ are connected to each other and every vertex in $T^2_{h-1}$ and $T^3_{h-1}$ due to the \ref{en: type 0 edges-treewidth} and \ref{en: type 2 edges-treewidth} edges.
Otherwise, as before, $A'$ is given by an admissible subgraph $A''$ and up to two copies of $T^w_{h-2}$. Set $X:=|A''|$ and $N:=|T^w_{h-2}|$.
If $|G| \leq 5N+X+2w+2$, we proceed exactly as in the case of Figure~\ref{fig:admissible-b-treewidth}, the only difference being that $R_{A}$ has three -- instead of two -- cliques as children. We are left with the case $5N+X+2w+2 < |G| \leq  8N+X+3w+3$. 
Thus, applying Lemma \ref{cor: sep treew2} yields a set of vertices $S_1 \subset G$ with $|S_1|=w+1$ and a partition $G_1,G_2,G_4, \overline{G}$ of $G\setminus S_1$, a set of vertices $S_2 \subset \overline{G}$ with $|S_2|=w+1$ and a partition $G_3, G_5,G_6$ of $\overline{G}\setminus S_2$ such that
\begin{equation}\label{part4-tw}
    |G_i| \leq 2N +\wavytriangle\left(\bigcup_{1\leq j<i} G_j\right)+I_i \quad \text{for } i\geq1,
    \quad 
    \text{and}
    \quad
    |G_1| +|G_2| \geq \frac{3}{2}N+X,
\end{equation}
where $I_6 = c$, if $|G| = 8N + X +2w+2+c$ for some $1\leq c \leq w+1$ and $I_i=0$ else.
Let $A^{(0)}=A$. In the following, for $1\leq i\leq 6$, we iteratively embed $G_i$ into $A^{(i-1)}$ and denote the resulting graph after the embedding $A^{(i)}$. We proceed in three main steps:
\begin{enumerate}
    \item\label{step1-tw}
    \begin{enumerate}
        \renewcommand{\labelenumii}{\alph{enumii})}
        \item\label{step1a-tw} Embed $G_1,\dots,G_{i_1}$, choosing $i_1$ minimal so that $R_{A'}$ are the first vertices in the eating order of $A^{(i_1)}$.
        \item\label{step1b-tw} Place $S_1$ at $R_{A'}$.
    \end{enumerate}
    \item\label{step2-tw}
    \begin{enumerate}
        \renewcommand{\labelenumii}{\alph{enumii})}
        \item\label{step2a-tw} Embed $G_{i_1+1},\dots, G_{i_2}$, choosing $i_2$ minimal so that $R_2=R_{T^2_{h-1}}$ are the first vertices in the eating order of $A^{(i_2)}$.
        \item\label{step2b-tw} Place $S_2$ at $R_{2}$.
    \end{enumerate}
    \item\label{step3-tw} Embed the remaining graphs $G_{i_2+1},\dots, G_6$. 
\end{enumerate}
We start with Step~\ref{step1-tw} and $i=1$.
Let $U^{(i)}$ be the induced subgraph of $A^{(i-1)}$ of size $2N+\wavytriangle(G_1\cup\dots \cup G_{i-1}) +w+1$ on $R_{A'}$ and the three subgraphs in $A^{(i-1)}$ rooted at level two that are eaten first. Then, $U^{(i)}$ is admissible by Observation~\ref{obs:2-tw}.
Further, by~\eqref{part4-tw}, $|G_i|\leq  |U^{(i)}|-(w+1)$ holds for $1\leq i \leq 5$. Using $U^{(i)}$, we embed $G_i$ into $A^{(i-1)}$ via the induction hypothesis.
We repeat this procedure until $R_{A'}$ are the first vertices in the eating order of the resulting graph $A^{(i_1)}$. 
Since $|G|\geq 5N+X+2w+3$, this process finishes with $i_1\leq 5$.
Next, we place $S_1$ at $R_{A'}$. 
This yields a proper embedding of $G_1,\dots,G_{i_1}$ and $S_1$ since the vertices in $R_{A'}$ are connected to every vertex in $A'$ and in $T^2_{h-1}$ thanks to the \ref{en: type 0 edges-treewidth} -- \ref{en: type 2 edges-treewidth} edges.

We proceed with Step~\ref{step2-tw} starting with $i=i_1+1$. 
Let $U^{(i)}$ be the induced subgraph of $A^{(i-1)}$ of size $2N+\wavytriangle(G_1\cup\dots \cup G_{i-1})+w+1$, consisting of the vertices $R_{2}$ together with the three subgraphs in $A^{(i-1)}$ rooted at level two that are eaten first. Then, $U^{(i)}$ is admissible by Observation~\ref{obs:2-tw}.
By~\eqref{part4-tw}, $|G_i|\leq  |U^{(i)}|-(w+1)$ holds unless $I_6>0$. The case $I_6>0$ only occurs when $|G|=8N+X+2w+2+c$, which in turn implies $i_2\leq 5$.
Thus, we embed $G_i$ into $A^{(i-1)}$ using the admissible graph $U^{(i)}$ until $R_{2}$ are the first vertices in the resulting graph $A^{(i_2)}$. 
By~\eqref{part4-tw}, $|G_1|+|G_2|\ge \frac{3}{2}N+X$, so that placing $S_2$ at $R_{2}$ provides a proper embedding of the first $i_2$ graphs and $S_1, S_2$ due to the \ref{en: type 0 edges-treewidth} -- \ref{en: type 3 edges-treewidth} edges. 

To conclude the proof, we address Step~\ref{step3-tw}. Observe that the graphs $G_1,\dots,G_{i_2}$ together with $S_1, S_2$ eat the first $|G_1|+ \dots+|G_{i_2}|+2w+2$ vertices in the eating order of $A$. Thus, by Observation~\ref{obs:1-tw}, the remaining graph is admissible. This finishes the proof by applying the induction hypothesis to $G_{i_2+1}\cup \dots \cup G_6$ and the remaining admissible graph.

\section{Acknowledgments}
This project was initiated during a research workshop in Buchboden, June 2024.
We want to thank Kalina Petrova for proposing the problem.
We are grateful to Anders Martinsson for valuable discussions.

\bibliographystyle{plain}
\bibliography{references}

\end{document}